\documentclass[oneside,a4paper,11pt,notitlepage]{article}
\usepackage[left=0.8in, right=0.8in, top=1.5in, bottom=1.5in]{geometry}
\usepackage{pifont}
\usepackage{makecell}
\usepackage[T1]{fontenc} 
\usepackage[utf8]{inputenc} 
\usepackage[english]{babel}
\usepackage{lipsum} 
\usepackage{lmodern}
\usepackage{amssymb}
\usepackage{amsthm}
\usepackage{bm}
\usepackage{mathtools}
\usepackage{braket}
\usepackage{esint}
\usepackage{lineno}
\newcommand{\abs}[1]{{\left|#1\right|}}

\newcommand{\curl}{\operatorname{curl}}

\usepackage{tabularx}
\usepackage{fontawesome}
\usepackage{booktabs}
\usepackage{graphicx}
\usepackage{tikz}
\usetikzlibrary{patterns}
\usepackage{multicol}
\usepackage{caption}
\usepackage{enumerate}
\usepackage{multicol}
\usepackage[skins,theorems]{tcolorbox}
\tcbset{highlight math style={enhanced,
		colframe=black,colback=white,arc=0pt,boxrule=1pt}}
\def\XXint#1#2#3{{\setbox0=\hbox{$#1{#2#3}{\int}$}
    \vcenter{\hbox{$#2#3$}}\kern-.5\wd0}}

\theoremstyle{definition}
\newtheorem{definizione}{Definition}[section]
\theoremstyle{plain}
\newtheorem{theorem}{Theorem}[section]

\newtheorem{lemma}[theorem]{Lemma}
\newtheorem{prop}[theorem]{Proposition}
\newtheorem{corollario}[theorem]{Corollary}
\theoremstyle{definition}
\newtheorem{esempio}{Example}[section]
\newtheorem{oss}[esempio]{Remark}

\renewcommand{\det}{\operatorname{det}}

\DeclareMathOperator{\R}{\mathbb{R}}

\DeclareMathOperator{\diam}{\, \textup{diam}}

\makeatletter
\newcounter{tempsection}

\newcommand{\saveformat}{%
  \setcounter{tempsection}{\value{section}}
  \let\oldthesection\thesection
}

\newcommand{\restoreformat}{%
  \let\thesection\oldthesection
  \renewcommand{\thesection}{\arabic{section}}
  \setcounter{section}{\value{tempsection}}
}
\makeatother

\makeatletter
\newcommand{\myfootnote}[2]{\begingroup
	\def\@makefnmark{}%
	\addtocounter{footnote}{-1}%
	\footnote{\textbf{#1} #2}
	\endgroup}
\makeatother

\usepackage{soul}

\definecolor{OliveGreen}{rgb}{0,0.6,0}

\usepackage{hyperref}
\hypersetup{linktoc=none, bookmarksnumbered, colorlinks=true, linkcolor=red}

\usepackage[normalem]{ulem}
\definecolor{DarkGreen}{rgb}{0,0.5,0.1} 

\newcommand\soutD{\bgroup\markoverwith
{\textcolor{DarkGreen}{\rule[.5ex]{2pt}{1pt}}}\ULon}
\newcommand\soutP{\bgroup\markoverwith
{\textcolor{blue}{\rule[.5ex]{2pt}{1pt}}}\ULon}
\newcommand{\Hm}[1]{\leavevmode{\marginpar{\tiny%
$\hbox to 0mm{\hspace*{-0.5mm}$\leftarrow$\hss}%
\vcenter{\vrule depth 0.1mm height 0.1mm width \the\marginparwidth}%
\hbox to
0mm{\hss$\rightarrow$\hspace*{-0.5mm}}$\\\relax\raggedright #1}}}
\newcommand{\noteD}[1]{\Hm{\textcolor{DarkGreen}{#1}}}
\title{\textbf{Spectral geometry of the magnetic $p$-Laplacian}}
\author{David Krej\v{c}i\v{r}\'{i}k and Rossano Sannipoli}
\date{}

\newcommand{\Addresses}{{
  \bigskip 
   \footnotesize 
 \noindent \textit{E-mail address}, D.~Krej\v{c}i\v{r}\'{i}k: \texttt{david.krejcirik@fjfi.cvut.cz} 
  
   \medskip

\noindent \textit{E-mail address}, R.~Sannipoli: \texttt{rossano.sannipoli@fjfi.cvut.cz} 
  
     \medskip 
\noindent \textsc{Department of Mathematics, Faculty of Nuclear Sciences and Physical Engineering, Czech Technical University in Prague, Trojanova 13, 120 00, Prague, Czech Republic.}\par\nopagebreak 

}} 

\begin{document}
\maketitle
\begin{abstract} 
We investigate the interplay between the geometry of a planar domain and the spectrum of the nonlinear $p$-Laplacian,
subject to a homogeneous magnetic field 
and Dirichlet boundary conditions.
First, we establish a lower bound for the bottom of the spectrum on the whole plane, extending the classical Landau-level estimate to the nonlinear regime. 
Second, we prove P\'olya-type bounds for the first eigenvalue on bounded convex planar domains in terms of the area, perimeter, minimal width and magnetic field strength, and show that they are asymptotically sharp along sequences of suitable thinning domains.
Finally, we post some remarks with open problems.
\\


\end{abstract}

\section{Introduction}
The principal objective of this paper is to initiate a study 
of geometrical aspects of spectral theory of the $p$-Laplacian
\begin{equation}\label{operator}
  u \mapsto - \nabla \cdot (|\nabla u|^{p-2} \nabla u)
  \qquad \mbox{with} \qquad
  1 < p < \infty
  \,,
\end{equation}
subject to magnetic fields.
The magnetic-free linear case $p=2$ is a traditional object of 
interest in spectral geometry, due to strong motivations 
in classical as well as quantum physics. 
More recently, with the advent of mesoscopic physics of nanostructures, 
the magnetic linear case has been considered, too.
In parallel, the magnetic-free nonlinear case $p \not= 2$ 
has equally been intensively studied in spectral geometry,
partly due to nonlinear phenomena in classical physics
but also due to the mathematical curiosity of the role 
of the nonlinear parameter~$p$ in spectral data.
In summary, the existing literature on 
the linear $p=2$ and possibly magnetic case,
or magnetic-free nonlinear case $p \not=2$, is enormous.  

On the other hand, the magnetic nonlinear cases $p \not=2$
represent a \emph{terra incognita},
which we would like to start to explore by this paper. 
The precedent voyages of exploration are very few and recent:
magnetic Hardy-type inequalities~\cite{CKLL2024}
(with follow-ups \cite{Chen-Tang_2025} and~\cite{Suragan})
and solutions to nonlinear Schr\"odinger equations~\cite{BFK}
(with follow-up \cite{BB2026}).
Moreover, the geometry of these works is restricted to 
the whole Euclidean space.  

In this paper, we are interested in spectral properties 
of the operator~\eqref{operator} possibly restricted 
to a Euclidean subspace 
and simultaneously subjected to a magnetic field.
We restrict to two dimensions 
and to a homogeneous magnetic field.
More specifically, we consider~\eqref{operator}  
acting on a simply connected domain $\Omega \subset \R^2$
and satisfying the Dirichlet boundary condition
$u=0$ on $\partial\Omega$.
Moreover, we replace the gradient~$\nabla$ 
by the magnetic gradient 

\begin{equation}\label{gradient}
  \nabla_A = \nabla + i A 
  \,,
\end{equation}
where $A: \Omega \to \R^2$ is a given smooth function
called the magnetic potential.
Physically more relevant is the magnetic field  
defined by 
\begin{equation*}
    B=\operatorname{curl}A=\partial_1A_2-\partial_2A_1
    \,,
\end{equation*}
which we assume to be a real-valued constant function.  

To make the long story short,
we are interested in the variational characterisation 
of the bottom of the spectrum given by 
\begin{equation}\label{Rayleigh}
    \lambda_p^B(\Omega)
    =\inf_{u\in W_{0,A}^{1,p}(\Omega)\setminus\{0\}}
    \frac{\displaystyle\int_\Omega |\nabla_Au|^p\,dx}
         {\displaystyle\int_\Omega |u|^p\,dx}
         \,,
\end{equation}
where the magnetic Sobolev space $W_{0,A}^{1,p}(\Omega)$
is defined as the closure of $C_c^\infty(\Omega)$
with respect to the norm 
\begin{equation*}
    \|u\|_{A,p}
    =\left(
    \|\nabla_A u\|_{L^p(\Omega)}^p + \|u\|_{L^p(\Omega)}^p
    \right)^{1/p}.
\end{equation*}
Although the Rayleigh-type quotient in~\eqref{Rayleigh} 
is written in terms of the vector potential~$A$, 
the spectral quantity $\lambda_p^B(\Omega)$
depends only on the magnetic field~$B$, 
due to the gauge invariance recalled 
in Section~\ref{sec:2}. 

Because of the Dirichlet boundary conditions,
we have the usual monotonicity with respect to the set inclusion.
Namely, if $\Omega_1\subset\Omega_2$, 
then
%
$
    \lambda_p^B(\Omega_2)\leq\lambda_p^B(\Omega_1)
$.    
%
Indeed, $ W_{0,A}^{1,p}(\Omega_1)$ can be considered 
as a subset of $ W_{0,A}^{1,p}(\Omega_2)$,
by extending a function from the former by zero outside of $\Omega_1$. 
 Furthermore, a simple change of variables yields the following scaling property
\begin{equation}\label{eq:scalingproperty}
    \lambda_p^B(t\Omega) = t^{-p} \, \lambda_p^{t^2B}(\Omega)
    \,, \qquad t>0
    \,,
\end{equation}
where $t\Omega := \{tx : x \in \Omega\}$.

In particular, in the magnetic-free case, one has
\begin{equation*}
    0\leq\lambda_p^0(\mathbb R^2)
    \leq\lambda_p^0(B_R)=R^{-p} \, \lambda_p^0(B_1),
\end{equation*}
where~$B_R$ is the ball of radius~$R$ centred at the origin of~$\R^2$.
As a consequence, taking $R\to +\infty$,
\begin{equation}\label{zero}
    \lambda_p^0(\mathbb R^2)=0
    \,.
\end{equation}
On the other hand,
in the magnetic linear case, 
it is well known that 
\begin{equation}\label{Landau}
  \lambda_2^B(\mathbb R^2) = |B|
  \,,
\end{equation}
which corresponds to the ground-state energy of 
the (explicitly solvable) Landau Hamiltonian. 
This is a particular circumstance of the diamagnetic 
effect in quantum mechanics:
the magnetic field has a tendency to enlarge the energy.
Mathematically, it is encoded in the diamagnetic inequality
(see, e.g., \cite[Thm.~7.2.1]{LL})
\begin{equation}\label{diamagnetic}
  |\nabla_A u| \geq |\nabla|u||
\end{equation}
valid pointwise for any $u \in W_{\mathrm{loc}}^{1,p}(\Omega)$. 

The first main result of this work quantifies 
the diamagnetic effect in the nonlinear case.
\begin{theorem}\label{thm:lower-bound-plane}
Let~$B$ be constant and $p>2$.
Then
\begin{equation}\label{Poincare}
         \lambda_p^B(\mathbb R^2)
    \geq
    \left(\frac{2}{p}\right)^{p/2}|B|^{p/2}.    
\end{equation}
\end{theorem}

The theorem does not follow from 
the diamagnetic inequality~\eqref{diamagnetic}, 
because of the magnetic-free result~\eqref{zero}. 
At the same time, the result does not follow 
from solving the spectral problem explicitly
(which does not seem to be feasible).
Instead, the proof is genuinely magnetic and relies on a regularised polar decomposition of the complex-valued function~$u$ in~\eqref{Rayleigh}, 
which separates its modulus from its phase without excluding the nodal set. 
The argument combines the curl constraint with an integration by parts, Young's inequality, H\"older's inequality, 
and a fairly non-trivial limiting procedure. 

The beauty of the Poincar\'e-type inequality~\eqref{Poincare} 
of Theorem~\ref{thm:lower-bound-plane} is that 
it naturally extends the classical Landau-level result~\eqref{Landau}
to the nonlinear case $p>2$.
Indeed, \eqref{Poincare}~reduces 
to the lower bound implied by~\eqref{Landau} when $p=2$. 
We do not know whether one actually has an equality in~\eqref{Poincare}. 
However, testing the Rayleigh quotient with the Gaussian ground state of the first Landau level in the linear case, 
we obtain the following upper bound which admits the same 
dependence on the strength of the magnetic field
and reduces to the upper bound implied by~\eqref{Landau} when $p=2$. 
(Moreover, here we allow for $p<2$ and even $p=1$.)
\begin{prop}\label{cor:two-sided-plane}
Let~$B$ constant and $p\geq 1$. 
Then
\begin{equation*}
    \lambda_p^B(\mathbb R^2)
    \leq
    \Gamma\!\left(\frac{p+2}{2}\right)
    \left(\frac{2}{p}\right)^{p/2}|B|^{p/2},
\end{equation*}
where~$\Gamma$ denotes the Euler's Gamma function.
\end{prop}

As an immediate consequence of the monotonicity with respect to the set inclusion and Theorem~\ref{thm:lower-bound-plane}, 
we get the following general lower bound for 
an arbitrary open set $\Omega$ (bounded or unbounded).

\begin{corollario}
Let~$B$ be constant and $p>2$.
For any simply connected domain $\Omega \subset \R^2$, 
one has
\begin{equation*}
    \lambda_p^B(\Omega) \ge \left(\frac{2}{p}\right)^{p/2}|B|^{p/2}.
\end{equation*}
\end{corollario}

Our next theorem describes the asymptotic behaviour in the strong magnetic field regime. Related asymptotic estimates 
in the linear case can be found in \cite{EKP2016}. 
Moreover, still in the linear case,
sharp asymptotic expansions have been obtained in the framework of semiclassical analysis in 
\cite{fournais2010spectral,Barbaroux-LeTreust-Raymond-Stockmeyer_2021}.
The following result demonstrates that the influence of the boundary disappears asymptotically and that bottom of the spectrum approaches
the bottom of the spectrum of the whole plane.
\begin{theorem}\label{thm:strongregime}
Let~$B$ be constant and $p>2$.
For any bounded simply connected domain $\Omega \subset \R^2$, 
one has
\begin{equation*}
\lim_{|B|\to+\infty}
\frac{\lambda_p^B(\Omega)}
{\lambda_p^B(\mathbb R^2)}
=1.
\end{equation*}
\end{theorem}

 We now focus on bounded convex domains~$\Omega$. 
 In the magnetic-free setting, 
 the problem of controlling the fundamental frequency
through elementary geometric quantities has a long history. 
For instance, denoting by~$|\Omega|$ and~$P(\Omega)$ 
the area and perimeter of~$\Omega$,
it holds (see \cite{polya1960} for the case $p=2$ and \cite{brasco_2020_principal_frequencies, gavitone_2014} otherwise)
\begin{equation}\label{eq:polyap}
    \lambda_p^0(\Omega)
    \leq \bigg(\frac{\pi_p}{2}\bigg)^p\frac{P(\Omega)^p}{|\Omega|^p}.
\end{equation}
Here $\pi_p$ stands for the lowest eigenvalue of 
the $p$-Laplacian in the interval $(0,1)$ (see for instance \cite{DPP2019})
given explicitly by 
\begin{equation*}
    \pi_p=2(p-1)^\frac{1}{p}\int_0^1\frac{dt}{(1-t^p)^{\frac{1}{p}}}
         =2\pi\frac{(p-1)^{\frac{1}{p}}}{p\sin(\pi/p)}.
\end{equation*}
The constant $(\pi_p/2)^p$ in \eqref{eq:polyap} is optimal, being approached by a sequence of thinning rectangles.
 
 There exists also a lower bound of the same type, that follows directly from the Hersch--Protter inequality (see~\cite{hersch1960frequence} and~\cite{protter1981lower} for the two--dimensional case and higher dimensional case, respectively, for $p=2$, and~\cite{brasco_inradius} for the other cases)
and a purely geometrical inequality proved in~\cite{fenchel_bonnesen}. The content of this lower bound will be clear in the proof of Proposition~\ref{prop:lowboundpolya}. 

These results illustrate two complementary aspects of
spectral geometry: P\'olya's construction uses distance function from the boundary as a trial
function to obtain sharp upper bounds, whereas the Hersch--Protter inequality controls the
eigenvalue from below through the largest ball contained in the domain. For the quantitative counterpart of this inequality and $p=2$ see \cite{AGS}; for other inequalities involving similar functionals see \cite{polya1960,makai,AMPS,AGS2026,BMS2026}.

In the magnetic problem, 
the diamagnetic inequality~\eqref{diamagnetic} transfers the classical
lower estimates immediately. 
The upper bounds require trial functions and gauges for which the ordinary gradient and
the magnetic potential can be controlled simultaneously. Our second main
result supplies this magnetic counterpart of P\'olya's estimate, including its
nonlinear $p$-Laplacian version.
In addition to the area~$|\Omega|$ 
and perimeter~$P(\Omega)$ of~$\Omega$,  
the magnetic field naturally leads to another geometric quantity,
namely the minimal width~$w_\Omega$ 
(for the precise definitions see Section \ref{sec:2}).
\begin{theorem}\label{thm:polya-bounds}
Let $B$ be constant and $p>1$. 
For any bounded convex domain $\Omega\subset\mathbb R^2$, 
one has
\begin{equation}\label{eq:lambdapB}
\lambda_p^B(\Omega)\leq
\begin{cases}
\displaystyle
\left(\frac{\pi_p}{2}\right)^p
\frac{P(\Omega)^p}{|\Omega|^p}
+\frac{|B|^p}{2^p}
\min\left\{
w_\Omega^p,
2^{2-p}\!\left[
\left(\frac{P(\Omega)}{4}\right)^p+
\left(\frac{w_\Omega}{2}\right)^p
\right]
\right\} & \mbox{if} \quad 1<p\leq2,\\
\displaystyle
\left[
\left(\frac{\pi_p}{2}\right)^2
\frac{P(\Omega)^2}{|\Omega|^2}
+\frac{B^2}{4}
\min\left\{
w_\Omega^2,
\left(\frac{P(\Omega)}{4}\right)^2+
\left(\frac{w_\Omega}{2}\right)^2
\right\}
\right]^{p/2} & \mbox{if} \quad p>2.
\end{cases}
\end{equation}
Moreover, the coefficient $(\pi_p/2)^p$ in the non-magnetic term cannot be
improved, and the estimate is asymptotically sharp along sequences of
thinning rectangles.
\end{theorem}

Here the two regimes arise from the concavity of $t^{p/2}$ for $p\leq2$ and
from Minkowski's inequality for $p>2$. The trial functions depend only on the
distance to the boundary; the coarea formula and P\'olya's change of variables
reduce the non-magnetic part to a one-dimensional problem, while a suitable
choice of gauge controls the magnetic part through weighted moments of the
coordinates. 
In the linear case $p=2$ (also new), 
Theorem~\ref{thm:polya-bounds}
reduces to
\begin{equation*}
    \lambda_1^B(\Omega)
    \leq
    \frac{\pi^2}{4}\frac{P(\Omega)^2}{|\Omega|^2}
    +\frac{B^2}{4}
    \min\left\{
        w_\Omega^2,
        \frac{P(\Omega)^2}{16}+\frac{w_\Omega^2}{4}
    \right\}.
\end{equation*}
The first term is precisely P\'olya's contribution, while the second controls
the additional magnetic energy through the perimeter and minimal width.

As we said previously, we complement these upper estimates 
with a universal lower bound. 
The following result is implied by the diamagnetism, 
the Hersch--Protter inequality, and the
relation between inradius, perimeter, and area. 

\begin{prop}\label{prop:lowboundpolya}
Let~$B$ be constant and $p>1$. 
For any bounded convex domain $\Omega\subset\mathbb R^2$, 
one has
\begin{equation*}
    \lambda_p^B(\Omega)
    \geq
    \bigg(\frac{\pi_p}{4}\bigg)^p\frac{P(\Omega)^p}{|\Omega|^p}.
\end{equation*}
The equality case is achieved, for instance, by sequences of thinning isosceles triangles.
\end{prop}

Combining this proposition with the P\'olya-type estimates
of Theorem~\ref{thm:polya-bounds},
we have managed to estimate the first nonlinear magnetic eigenvalue
between two explicit quantities involving only elementary geometric data and the field strength.

The paper is organised as follows.
In Section~\ref{sec:2} we collect our notation and the preliminary material about the gauge invariance and convex geometry that will be used throughout the paper.
Section~\ref{sec:3} contains the proof of Theorem~\ref{thm:lower-bound-plane} and Proposition~\ref{cor:two-sided-plane}.
Moreover, we investigate the strong magnetic field regime
and establish Theorem~\ref{thm:strongregime}. 
Section~\ref{sec:4} is devoted to the proof of 
Theorem~\ref{thm:polya-bounds} and Proposition~\ref{prop:lowboundpolya}.
We eventually collect some open problems in Section~\ref{sec:5}.

\section{Preliminaries}\label{sec:2}

\subsection{Basic notations and a complex algebra}
 
Throughout this article, 
given vectors $x,y \in \R^n$ with $n \geq 1$,
$|x|$ and $x \cdot y$ denote 
the standard Euclidean norm of~$x$
and scalar product of~$x$ and~$y$, respectively.

 If $\Omega$ is a measurable set of~$\R^n$, 
 then $|\Omega|$ denotes its Lebesgue measure. 
 By $\mathcal{H}^k(\Omega)$, for $k\in [0,n)$, we denote the $k$-dimensional Hausdorff measure of~$\Omega$ in $\mathbb{R}^n$. The perimeter of $\Omega$ in $\mathbb{R}^n$ will be denoted by $P(\Omega)$ and, if $P(\Omega)<\infty$, we say that $\Omega$ is a set of finite perimeter. When $\Omega$ is bounded, open and convex, then $\Omega$ is a set of finite perimeter and $P(\Omega)=\mathcal{H}^{n-1}(\partial\Omega)$. For results relative to the sets of finite perimeter and for the coarea formula, we refer to \cite{ambrosio2000functions,maggi2012sets}.
 
In this paper, we typically restrict to $n=2$. 
 Given a smooth function $f : \R^2 \to \mathbb{C}$, we will denote by $\operatorname{Re}(f)$ and $\operatorname{Im}(f)$ its real and imaginary parts, respectively. Moreover we write 
\begin{equation*}
    \nabla^\perp f = (\partial_y f, -\partial_x f).
\end{equation*}
In addition, for a smooth vector field $F:\R^2 \to \mathbb{C}^2$,
one has
\begin{equation}\label{eq:curlproduct}
    \curl(fF) = f\curl F + \nabla f\times F= f\curl F- \nabla^\perp f\cdot F
    \,,
\end{equation}
where~$\times$ denotes the vector product. We now prove some useful identities that will be necessary for the proof of Theorem~\ref{thm:lower-bound-plane}. \begin{lemma}\label{lem:A2}
    Let $u\in C^2(\mathbb R^2,\mathbb C)$, then we have the following equalities:
\begin{align}
    &\nabla^\perp |u|^2 = 2 \operatorname{Re}(\overline{u}\nabla^\perp u),\label{eq:A2.1} \\
    &\operatorname{Re}(\overline{u}\nabla^\perp u)\cdot \operatorname{Im}(\overline{u}\nabla u) = -|u|^2 \det (\nabla\operatorname{Re}(u),\nabla\operatorname{Im}(u)), \label{eq:A2.2}  \\
   &\curl(\operatorname{Im}(\overline{u}\nabla u))= 2 \det (\nabla\operatorname{Re}(u),\nabla\operatorname{Im}(u)). \label{eq:A2.3}
\end{align}
\end{lemma}

\begin{proof}
    Let $u = a+ib$, with $a= \operatorname{Re}(u),b=\operatorname{Im}(u) \in C^2(\mathbb R^2, \mathbb R)$. Then $|u|^2= a^2+b^2$ and so
\begin{equation*}
    \nabla |u|^2 = 2(aa_x+bb_x, aa_y+bb_y) \implies \nabla^\perp |u|^2 = 2(aa_y+bb_y, -aa_x-bb_x).
\end{equation*}
Let us now find the components of this gradient and the perpendicular gradient to write down the real and imaginary part of $\overline{u}\nabla u$. Then $\overline{u}= a-ib$ and $\nabla u = \nabla a+ i\nabla b$. Hence
\begin{equation*}
    \overline{u}\nabla u = (a-ib)(\nabla a+i\nabla b) = a\nabla a + b\nabla b + i(a\nabla b-b\nabla a).
\end{equation*}
So we get
\begin{equation*}
    \begin{split}
    &\operatorname{Re}(\overline{u}\nabla u) = a\nabla a + b\nabla b= (a a_x+ b b_x,a a_y+b b_y), \\
    &\operatorname{Im}(\overline{u}\nabla u) = a\nabla b-b\nabla a = (a b_x- b a_x,a b_y-b a_y).    
    \end{split}
\end{equation*}
Moreover if we consider $\overline{u}\nabla^\perp u$, we arrive similarly to
\begin{equation*}
    \begin{split}
    &\operatorname{Re}(\overline{u}\nabla^\perp u) = a\nabla^\perp a + b\nabla^\perp b= (a a_y+ b b_y,-a a_x-b b_x), \\
    &\operatorname{Im}(\overline{u}\nabla^\perp u) = a\nabla^\perp b-b\nabla^\perp a = (a b_y- b a_y,-a b_x+b a_x).     
    \end{split}
\end{equation*}
These computation show firstly equation~\eqref{eq:A2.1}.
Now
\begin{equation*}
    \begin{split}
\operatorname{Re}(\overline{u}\nabla^\perp u)\cdot \operatorname{Im}(\overline{u}\nabla u) &= (a a_y+b b_y,-a a_x-b b_x)\cdot (a b_x- b a_x,a b_y-b a_y) \\
&=(a a_y+ b b_y)(a b_x- b a_x)+ (-a a_x-b b_x)(a b_y-b a_y)\\
&= -a^2(a_xb_y-a_yb_x)- b^2(a_xb_y-a_yb_x)\\
&= -(a^2+b^2)(a_xb_y-a_yb_x)= -(a^2+b^2)\det(\nabla a,\nabla b).
\end{split}
\end{equation*}
Therefore we have~\eqref{eq:A2.2}.
Finally 
\begin{equation*}
    \begin{split}
\curl(\operatorname{Im}(\overline{u}\nabla u))&=\partial_x(a b_y-b a_y) -  \partial_y(a b_x- b a_x)\\
&= a_xb_y+ab_{yx}-b_xa_{y}-ba_{yx}-a_yb_x-ab_{xy}+b_ya_x+ba_{xy}.
    \end{split}
\end{equation*}
Since $a,b$ are smooth, we have that $a_{xy}=a_{yx}$ and $b_{xy}=b_{yx}$, so that
\begin{equation*}
    \curl(\operatorname{Im}(\overline{u}\nabla u))= 2(a_xb_y-a_yb_x) = 2\det(\nabla a,\nabla b),
\end{equation*}
and so~\eqref{eq:A2.3}.
\end{proof}

 \subsection{Elements of convex geometry}
  Let $\Omega$ be a bounded, open and convex set of $\mathbb{R}^n$. The support function of $\Omega$ is defined by
$
    h_\Omega(y)=\sup_{x\in \Omega}\left(x\cdot y\right)
$, 
$y\in \mathbb{R}^n$.
The width of $\Omega$ in the direction $y \in \mathbb{R}$ is defined as 
$    \omega_{\Omega}(y)=h_{\Omega}(y)+h_{\Omega}(-y) $,
 and the minimal width of $\Omega$ is the minimal distance between any two parallel support hyperplanes, i.e.
\begin{equation*}
    w_\Omega=\min\{  \omega_{\Omega}(y)\,:\; y\in\mathbb{S}^{n-1}\}.
\end{equation*}
The inradius of $\Omega$ is defined by
 \begin{equation}
     \label{inradiuss}
     R_
     \Omega=\sup\{r\in \mathbb R: B_r(x)\subset \Omega, x\in \Omega\},
 \end{equation}
and 
$  
\diam(\Omega)
= \sup_{x,y\in\Omega}|x-y|
$ 
stands for the diameter of~$\Omega$.

\begin{definizione}
    We say that a family $\{\Omega_\ell\}_{\ell > 0}$
    of non-empty, bounded, open and convex sets of $\mathbb{R}^2$
    is a family of  thinning domains if
    \begin{equation*}
        \dfrac{w_{\Omega_\ell}}{\diam(\Omega_\ell)}\xrightarrow{\ell \to 0}0.
    \end{equation*}

\end{definizione}

Particular families of thinning domains of our interests are as follows.
\begin{definizione} 
Let~$a$ be a positive number. 
The family $\{\Omega_a\}_{a > 0}$ defined as  $$\Omega_a = \big(0,1\big)\times \big(0,a\big)$$ is called a family  of \emph{thinning rectangles}. Given the point $V_a=(0,a)\in \mathbb R^2$ and the interval $I=(-1,1)$, the family $\{T_a\}_{a>0}$ defined as 
  \begin{equation*}
      T_a = conv\{I, V_a\}, 
  \end{equation*}
where $conv$ denotes the convex hull, 
is called a family of \emph{thinning isosceles triangles}.
\end{definizione}

The distance function from the boundary of~$\Omega$ 
is defined by
 \begin{equation*}
    d: \Omega \to [0,+\infty) :
    \left\{ x \mapsto \inf_{y\in\partial\Omega}\abs{x-y} \right\}
    .
 \end{equation*}
The inradius $R_\Omega$ is its maximum value. The distance function is concave, as a consequence of the convexity of $\Omega$ and for a.e.\ $x \in \Omega$ we have that $|\nabla d|=1$. For $t \in [0, R_\Omega]$, the inner parallel set $\Omega_t$ is the superlevel set of the distance function from the boundary at level $t$, i.e.
\begin{equation*}
    \Omega_t =\{x\in \Omega: d(x) > t\}.
\end{equation*}
In particular, we denote by
$ \mu(t) = |\Omega_t| $ and $P(t) = P(\Omega_t)$
the measure and perimeter of the inner parallel sets, respectively.

\subsection{Gauge invariance and symmetries}
What follows can be found in \cite{fournais2010spectral}.
Let $A,\tilde A : \Omega \to \R^2$ 
be two smooth magnetic potentials 
generating the same (not necessarily constant) magnetic field, namely
$B = \curl A = \curl\tilde A$. 
Consequently, $\curl(\tilde A-A)=0$ in~$\Omega$.
If \(\Omega\) is simply connected, 
then Poincaré's lemma ensures that there exists 
a smooth function $\varphi:\Omega \to \R$ 
such that $\tilde A=A+\nabla\varphi.
$
 In particular, if \(u\in W^{1,p}_{0,A}(\Omega)\), 
 then $v=e^{-i\varphi}u \in W^{1,p}_{0,\tilde A}(\Omega)$ 
 and $\nabla_{\tilde A}v= e^{-i\varphi}\nabla_Au
$. 
Consequently,
$|\nabla_{\tilde A}v|
=
|\nabla_Au|,$ and $
|v|
=
|u|$, which implies that the corresponding Rayleigh quotients 
in~\eqref{Rayleigh} coincide. It follows that, on simply connected domains~$\Omega$, the bottom of the spectrum $\lambda_p^B(\Omega)$
indeed depends only on the magnetic field~\(B\), and not on the particular choice of the magnetic potential.

For the homogeneous (i.e., constant) field~$B$,
two of the most known gauges in literature are 
the \textit{symmetric}
(or \emph{Poincar\'e} or \emph{transverse}) gauge
\begin{equation}\label{symmetric}
    A(x) = \frac{B}{2}(-x_2,x_1),
\end{equation}
and the \textit{Landau} gauge
\begin{equation*}
    A(x) = (0,Bx_1) \qquad \text{or}\qquad A(x) = -(Bx_2,0).
\end{equation*}
Below we shall also consider an interpolation 
between these standard choices, see~\eqref{affine}.

If $\tilde A = -A$, then $\curl \tilde A = - \curl A$, 
while $|\nabla_Au|=|\nabla_{\tilde{A}}u|$, and therefore 
\begin{equation}\label{eq:B-B}
    \lambda_p^B(\Omega)= \lambda_p^{-B}(\Omega). 
\end{equation}

A similar argument also works with the translations. 
Let $x_0 \in \mathbb R^2$. Let $\Omega_{x_0}= \Omega-\{x_0\}$. Let us assume that we are in the symmetric gauge. In particular
\begin{equation*}
    A(x-x_0) = A(x)-A(x_0).
\end{equation*}
Since $A(x_0)$ is a point, it means that if we consider $\varphi = A(x_0)\cdot x$, then $\nabla \varphi = A(x_0)$. Therefore $A(x-x_0)= A(x) + \nabla\varphi$ and a change of variables in~\eqref{Rayleigh} yields
\begin{equation}\label{eq:gaugetranslations}
    \lambda_p^B(\Omega_{x_0})= \lambda_p^B(\Omega).
\end{equation}

\section{The whole plane}\label{sec:3}
In this section, we consider the case of the homogeneous field
in the whole plane. 
In particular, we establish Theorem~\ref{thm:lower-bound-plane}
and Proposition~\ref{cor:two-sided-plane}.

\subsection{An idealised proof of the lower bound of Theorem~\ref{thm:lower-bound-plane}}

For Theorem~\ref{thm:lower-bound-plane}, 
which is the main and more advanced result, 
we first treat the idealised situation in which the function \(u\) in the Rayleigh quotient~\eqref{Rayleigh} admits a polar decomposition with a sufficiently smooth modulus and phase. Although restrictive, this setting captures the main mechanism of the proof and provides the blueprint for the general case.
\begin{prop}\label{prop:low_bound_plane_regular}
Let $B$ be constant and $p > 2$. 
Let $u \in C^\infty_c(\mathbb R^2, \mathbb C)$ be such that $u = \rho e^{i\phi}$ with $\rho\in C^\infty_c(\mathbb R^2, \mathbb R)$ and $\phi\in C^2(\mathbb R^2,\mathbb R)$. Then
\begin{equation*}
  \frac{\displaystyle \int_{\R^2}
  |\nabla_Au|^p}{\displaystyle\int_{\mathbb R^2}|u|^p} \geq \left(\frac{2}{p}\right)^{p/2} |B|^{p/2}.
\end{equation*}   
\end{prop}
\begin{proof}
Since $u=\rho e^{i\phi}$, one has
$
  \nabla u = e^{i\phi}\bigl(\nabla \rho + i\rho \nabla \phi\bigr).
$  
Hence
\begin{equation}\label{eq:smoothcaseDA}
  \nabla u + iAu
  = e^{i\phi}\bigl[\nabla \rho + i\rho(\nabla \phi + A)\bigr].
\end{equation}
Now, since $|e^{i\phi}| = 1$, we have
\begin{align*}
  |(\nabla + iA)u|^2
  &= |e^{i\phi}|^2
     \bigl|\nabla \rho + i\rho(\nabla \phi + A)\bigr|^2 \\
  &= |\nabla \rho|^2
     + \rho^2 |\nabla \phi + A|^2
     + 2\operatorname{Re}\langle \nabla \rho,\overline{i\rho(\nabla \phi + A)}\rangle.
\end{align*}
But since $\nabla \rho$ and $\nabla \phi + A$ are real-valued, then
$
  \operatorname{Re}\langle \nabla \rho,\overline{i\rho(\nabla \phi + A)}\rangle = 0,
$  
and therefore
\begin{equation*}
  |(\nabla + iA)u|^2
  = |\nabla \rho|^2 + \rho^2|\nabla \phi + A|^2.
\end{equation*}
For simplicity, since $\mathrm{curl} \nabla\phi = 0$, let us denote
$
  V = \nabla\phi + A.
$  
Then $\mathrm{curl} V = \mathrm{curl} A = B$, and hence
$
  |(\nabla + iA)u|^2
  = |\nabla \rho|^2 + \rho^2 |V|^2.
$  
Consequently,
\begin{equation*}
  \int_{\R^2}|(\nabla + iA)u|^p
  = \int_{\R^2}\bigl(|\nabla \rho|^2 + \rho^2|V|^2\bigr)^{p/2}.
\end{equation*}
Now, since $\mathrm{curl} V = B$, we have
\begin{equation*}
  B\int_{\R^2}\rho^p
  = \int_{\R^2}\rho^p(\partial_x V_2-\partial_y V_1).
\end{equation*}
Since $\rho$ has compact support in $\R^2$, 
an integration by parts yields
\begin{equation*}
  \int_{\R^2}\rho^p\partial_x V_2
  = -\int_{\R^2}V_2\partial_x(\rho^p)
  \qquad\mbox{and}\qquad
  -\int_{\R^2}\rho^p\partial_y V_1
  = \int_{\R^2}V_1\partial_y(\rho^p).
\end{equation*}
Hence
\begin{align*}
  B\int_{\R^2}\rho^p
  = p\int_{\R^2}\rho^{p-1}\bigl(V_1\partial_y\rho -V_2\partial_x\rho\bigr) 
  = p\int_{\R^2}\rho^{p-1}\bigl(V\cdot \nabla^\perp\rho\bigr).
\end{align*}
Thus, considering the modulus,
\begin{equation*}
  |B|\int_{\R^2}\rho^p
  \leq p\int_{\R^2}\rho^{p-1}|V|\,|\nabla \rho|,
\end{equation*}
where we have used that $|\nabla^\perp\rho| = |\nabla\rho|$. 
Now we write
$
  \rho^{p-1}|V|\,|\nabla\rho|
  = \rho^{p-2}|\nabla\rho|\,\rho|V|.
$  
By Young's inequality
$
  2ab \leq a^2 + b^2,
$  
with $a = |\nabla\rho|$ and $b = \rho|V|$, we get
\begin{equation*}
  \rho^{p-1}|V|\,|\nabla\rho|
  \leq \frac{\rho^{p-2}}{2}
  \bigl(|\nabla\rho|^2 + \rho^2|V|^2\bigr).
\end{equation*}
Hence
\begin{equation*}
  |B|\int_{\R^2}\rho^p
  \leq \frac{p}{2}\int_{\R^2}\rho^{p-2}
  \bigl(|\nabla\rho|^2 + \rho^2|V|^2\bigr).
\end{equation*}
Applying H\"older's inequality with exponents $p/2$ and its conjugate $q = p / (p-2)$, 
we have
\begin{align*}
  |B|\int_{\R^2}\rho^p
  &\leq \frac{p}{2}
  \int_{\R^2}\rho^{p-2}\bigl(|\nabla\rho|^2 + \rho^2|V|^2\bigr) \\
  &\leq \frac{p}{2}
  \left(\int_{\R^2}\rho^p\right)^{(p-2)/p}
  \left(\int_{\R^2}\bigl(|\nabla\rho|^2 + \rho^2|V|^2\bigr)^{p/2}\right)^{2/p}.
\end{align*}
Taking the $p/2$ power, we obtain
\begin{equation*}
  |B|^{p/2}\left(\int_{\R^2}\rho^p\right)^{p/2}
  \leq \left(\frac{p}{2}\right)^{p/2}
  \left(\int_{\R^2}\rho^p\right)^{(p-2)/2}
  \int_{\R^2}\bigl(|\nabla\rho|^2 + \rho^2|V|^2\bigr)^{p/2}.
\end{equation*}
This gives
\begin{align*}
  \int_{\R^2}\bigl(|\nabla\rho|^2 + \rho^2|V|^2\bigr)^{p/2}
  \geq
  \left(\frac{2}{p}\right)^{p/2}|B|^{p/2}
  \frac{\displaystyle\left(\int_{\R^2}\rho^p\right)^{p/2}}
       {\displaystyle\left(\int_{\R^2}\rho^p\right)^{(p-2)/2}} = \left(\frac{2}{p}\right)^{p/2}|B|^{p/2}
  \int_{\R^2}\rho^p.
\end{align*}
Therefore
\begin{equation*}
  \frac{\displaystyle
  \int_{\R^2}\bigl(|\nabla\rho|^2 + \rho^2|V|^2\bigr)^{p/2}}
  {\displaystyle \int_{\R^2}\rho^p}
  \geq \left(\frac{2}{p}\right)^{p/2}|B|^{p/2}.
\end{equation*}
This concludes the proof of Proposition~\ref{prop:low_bound_plane_regular}.
\end{proof}

\subsection{A regularisation idea for 
the proof of Theorem~\ref{thm:lower-bound-plane}}

As we pointed out before, Proposition~\ref{prop:low_bound_plane_regular} is proved under very specific hypotheses. In general, if $u\in C^\infty_c(\mathbb R^2, \mathbb C)$, it is not guaranteed that the phase $\phi$ is regular enough. To overcome this issue we will rather consider the regularised decomposition  
\begin{equation}\label{decomposition}
    u = \rho_\varepsilon \xi_\varepsilon,
\end{equation}
where $\rho_\varepsilon= \sqrt{|u|^2+\varepsilon^2}$ and $\xi_\varepsilon= u/\rho_\varepsilon$. 

The main difficulty is that the vector field \(V=\nabla\phi+A\), which plays a crucial role in the proof of Proposition~\ref{prop:low_bound_plane_regular}, is no longer available. Indeed, the phase may fail to be sufficiently regular and it is not even defined on the nodal set of \(u\). Our strategy is therefore to construct a smooth vector field that agrees with \(V\) whenever the latter is well defined and can be used in its place in the general case. The proof of Proposition~\ref{prop:low_bound_plane_regular} naturally suggests how such a vector field should be defined. 
If $u=\rho e^{i\theta}$ and $\rho,\theta$ are regular enough, then
\begin{equation*}
  \nabla_Au=e^{i\theta}\bigl(\nabla\rho+i\rho(\nabla\theta+A)\bigr)
  =e^{i\theta}\bigl(\nabla\rho+i\rho V\bigr).
\end{equation*}
Multiplying by $\overline u=\rho e^{-i\theta}$ gives
\begin{equation*}
  \overline u\nabla_Au
  =\rho\nabla\rho+i\rho^2V.
\end{equation*}
Since $\rho$ and $\theta$ are real-valued, $\operatorname{Im}(\rho\nabla\rho)=0$, and hence considering the imaginary part both sides we get (whenever $\rho \neq 0)$
\begin{equation*}
  V=\frac{\operatorname{Im}(\overline u\nabla_Au)}{\rho^2}.
\end{equation*}
This suggests the regularised definition
\begin{equation}\label{eq:Vepsilon}
  V_\varepsilon
  =\frac{\operatorname{Im}(\overline u\nabla_Au)}{\rho_\varepsilon^2}
  =\frac{\operatorname{Im}(\overline u\nabla_Au)}{|u|^2+\varepsilon^2},
\end{equation}
which is smooth and therefore we can work with it in a classical sense. We can also write this expression in a more explicit way. Since
\begin{equation*}
  \overline u\nabla_Au
  =\overline u(\nabla u+iAu)
  =\overline u\nabla u+iA|u|^2,
\end{equation*}
we have
\begin{equation*}
  \operatorname{Im}(\overline u\nabla_Au)
  =\operatorname{Im}(\overline u\nabla u)+A|u|^2.
\end{equation*}
Therefore
\begin{equation}\label{eq:Vepsilon2}
  V_\varepsilon
  =\frac{\operatorname{Im}(\overline u\nabla u)+A|u|^2}{|u|^2+\varepsilon^2}
  =A +\frac{\operatorname{Im}(\overline u\nabla u)-\varepsilon^2A}{|u|^2+\varepsilon^2}.
\end{equation}

\subsection{The proof of Theorem~\ref{thm:lower-bound-plane}}

For the proof it will be sufficient to consider $u\in C^\infty_c(\mathbb R^2; \mathbb C)$, since an approximation argument will ensure the validity of the result for a generic $u \in W^{1,p}_{0,A}(\R^2)$.

Let $\varepsilon>0$ and let us consider 
the decomposition~\eqref{decomposition}.
Notice that, since $u\in C_c^\infty(\R^2,\mathbb{C})$ and $\rho_\varepsilon>0$, then
$ \xi_\varepsilon\in C^\infty_c(\R^2,\mathbb{C})$. Moreover
\begin{equation*}
  |\xi_\varepsilon|=\frac{|u|}{\sqrt{|u|^2+\varepsilon^2}}
  \leq 1.
\end{equation*}
Let us write its magnetic gradient:
\begin{align*}
  \nabla_Au
  =(\nabla+iA)(\rho_\varepsilon \xi_\varepsilon) =\xi_\varepsilon\nabla\rho_\varepsilon
    +\rho_\varepsilon(\nabla \xi_\varepsilon+iA\xi_\varepsilon).
\end{align*}
We want to make appear something similar to \eqref{eq:smoothcaseDA} on the right-hand side. We do it by adding and subtracting $i\xi_\varepsilon \rho_\varepsilon V_\varepsilon$, where $V_\varepsilon$ is the regularised version of the vector appearing in Proposition \ref{prop:low_bound_plane_regular}, defined in \eqref{eq:Vepsilon}. Hence
\begin{equation*}
  \nabla_Au = \xi_\varepsilon\bigl(\nabla\rho_\varepsilon+i\rho_\varepsilon V_\varepsilon\bigr)+\rho_\varepsilon\bigl(\nabla \xi_\varepsilon+i(A-V_\varepsilon)\xi_\varepsilon\bigr)=\xi_\varepsilon X_\varepsilon+ E_\varepsilon,
\end{equation*}
 where we denote by 
\begin{equation*}
  X_\varepsilon=\nabla\rho_\varepsilon+i\rho_\varepsilon V_\varepsilon, \qquad E_\varepsilon= \rho_\varepsilon\bigl(\nabla \xi_\varepsilon+i(A-V_\varepsilon)\xi_\varepsilon\bigr).
\end{equation*}
We now compute the error term $E_\varepsilon$. We have
\begin{equation*}
  \nabla \xi_\varepsilon
  =\nabla\left(\frac{u}{\rho_\varepsilon}\right)
  =\frac{\nabla u}{\rho_\varepsilon}
    -\frac{u\nabla\rho_\varepsilon}{\rho_\varepsilon^2}=\frac{\nabla u}{\rho_\varepsilon}-\frac{u\operatorname{Re}(\overline u\nabla u)}{\rho_\varepsilon^3},
\end{equation*}
where we used (see \cite[Prop.~2.1.2]{fournais2010spectral})
\begin{equation*}
  \nabla\rho_\varepsilon
  =\frac{\operatorname{Re}(\overline u\nabla u)}{\rho_\varepsilon}.
\end{equation*}
Hence, recalling the decomposition~\eqref{decomposition},
\begin{equation*}
E_\varepsilon
  =\rho_\varepsilon\bigl(\nabla \xi_\varepsilon+i(A-V_\varepsilon)\xi_\varepsilon\bigr) \\
  =\nabla u-\frac{u\operatorname{Re}(\overline u\nabla u)}{\rho_\varepsilon^2}
    +iAu-iV_\varepsilon u.
\end{equation*}
Using the definition of $V_\varepsilon$,
\begin{equation*}
    E_\varepsilon= \nabla u-u\frac{\operatorname{Re}(\overline u\nabla u)+i\operatorname{Im}(\overline u\nabla u)}{\rho_\varepsilon^2}
     +i\frac{\varepsilon^2uA}{\rho_\varepsilon^2}.
\end{equation*}
But $\operatorname{Re}(\overline u\nabla u)+i\operatorname{Im}(\overline u\nabla u)= \overline u\nabla u$ and $u\overline{u}=|u|^2$, so we have
\begin{equation*}
 E_\varepsilon= \nabla u-\frac{|u|^2}{\rho_\varepsilon^2}\nabla u
    +\frac{\varepsilon^2uA}{\rho_\varepsilon^2} 
  =\frac{\varepsilon^2}{|u|^2+\varepsilon^2}(\nabla+iA)u.
\end{equation*}
Eventually
\begin{equation}\label{eq:Evarepsilon}
    E_\varepsilon=\frac{\varepsilon^2}{|u|^2+\varepsilon^2}\nabla_Au.
\end{equation}
Moreover
\begin{equation*}
  \nabla_Au=\xi_\varepsilon X_\varepsilon+E_\varepsilon
  =\xi_\varepsilon X_\varepsilon+\frac{\varepsilon^2}{|u|^2+\varepsilon^2}\nabla_Au.
\end{equation*}
Hence
\begin{equation*}
  \left(1-\frac{\varepsilon^2}{|u|^2+\varepsilon^2}\right)\nabla_Au
  =\xi_\varepsilon X_\varepsilon.
\end{equation*}
Since 
\begin{equation*}
    1-\frac{\varepsilon^2}{|u|^2+\varepsilon^2}=\frac{|u|^2}{|u|^2+\varepsilon^2} = |\xi_\varepsilon|^2,
\end{equation*}
we have
\begin{equation}\label{eq:identity1}
  \xi_\varepsilon X_\varepsilon=|\xi_\varepsilon|^2\nabla_Au.
\end{equation}
Therefore
\begin{align*}
  |\nabla_Au|^2
  =|\xi_\varepsilon X_\varepsilon+E_\varepsilon|^2 =|\xi_\varepsilon|^2|X_\varepsilon|^2+|E_\varepsilon|^2
    +2\operatorname{Re}\bigl(\xi_\varepsilon X_\varepsilon\cdot\overline{E_\varepsilon}\bigr).
\end{align*}
Using \eqref{eq:Evarepsilon} and \eqref{eq:identity1}, we get
\begin{equation*}
  \operatorname{Re}\bigl(\xi_\varepsilon X_\varepsilon\cdot\overline{E_\varepsilon}\bigr)
  =\frac{\varepsilon^2|u|^2}{(|u|^2+\varepsilon^2)^2}|\nabla_Au|^2
 \ge 0.
\end{equation*}
This yields
$
    |\nabla_Au|^2 \ge|\xi_\varepsilon|^2|X_\varepsilon|^2= |\xi_\varepsilon|^2|\nabla\rho_\varepsilon+i\rho_\varepsilon V_\varepsilon\bigr|^2,
$    
and therefore
\begin{equation}\label{eq:importantestimate}
    \int_{\mathbb R^2}|\nabla_A u|^p\,dx \ge \int_{\mathbb R^n} |\xi_\varepsilon|^p|\nabla\rho_\varepsilon+i\rho_\varepsilon V_\varepsilon\bigr|^p\,dx.
\end{equation}
Let
$
  V_\varepsilon=(V_\varepsilon^1,V_\varepsilon^2).
$  
Using the explicit formula \eqref{eq:Vepsilon2} for $V_\varepsilon$, we have
\begin{equation*}
  \curl V_\varepsilon
  =B+\curl\left(
  \frac{\operatorname{Im}(\overline u\nabla u)-\varepsilon^2A}{|u|^2+\varepsilon^2}
  \right).
\end{equation*}
If we denote
\begin{equation*}
  W_\varepsilon
  =\frac{\operatorname{Im}(\overline u\nabla u)-\varepsilon^2A}{|u|^2+\varepsilon^2},
\end{equation*}
then
\begin{equation}\label{eq:differencecurl}
  B=\curl V_\varepsilon-\curl W_\varepsilon.
\end{equation}
Since $|\xi_\varepsilon|\leq1$, we have
\begin{equation*}
  B\int_{\R^2}|u|^p
  =B\int_{\R^2}\rho_\varepsilon^p|\xi_\varepsilon|^p
  \leq B\int_{\R^2}\rho_\varepsilon^p.
\end{equation*}
We want to write $B$ as in \eqref{eq:differencecurl} and then integrate by parts, but we need to slightly modify $\rho_\varepsilon$, which is not compactly supported in $\mathbb R^2$. Therefore we write
\begin{equation}\label{regularisation}
  \rho_\varepsilon=r_\varepsilon+\varepsilon,
  \qquad
  r_\varepsilon=\sqrt{|u|^2+\varepsilon^2}-\varepsilon.
\end{equation}
We notice that $r_\varepsilon\in C_c^\infty(\mathbb R^2,\R)$ and $0< r_\varepsilon\leq |u|$. Moreover, since $u$ is compactly supported, a Taylor expansion yields
\begin{equation*}
  B\int_{\R^2}\rho_\varepsilon^p= B\int_{\mathbb R^2}(r_\varepsilon + \varepsilon)^p
  =B\int_{\R^2}r_\varepsilon^p+o(1)
\end{equation*}
as $\varepsilon \to 0$.
Recalling that $B= \curl V_\varepsilon-\curl W_\varepsilon$, we have that 
\begin{equation*}
     B\int_{\R^2}\rho_\varepsilon^p= B\int_{\R^2}r_\varepsilon^p+o(1)= \int_{\R^2}r_\varepsilon^p\curl V_\varepsilon-\int_{\R^2}r_\varepsilon^p\curl W_\varepsilon + o(1).
\end{equation*}
Therefore, considering the modulus, we have that 

\begin{equation}\label{submerged}
    |B|\int_{\R^2}|u|^p\le 
    \underbrace{\bigg|
    \int_{\R^2}r_\varepsilon^p\curl V_\varepsilon
    \bigg|}_{I_{1,\varepsilon}}
    +\underbrace{\bigg|
    \int_{\R^2}r_\varepsilon^p\curl W_\varepsilon
    \bigg|}_{I_{2,\varepsilon} }
    +o(1).  
\end{equation}

The proof of the theorem will be concluded immediately
after establishing this submerged lemma.
\begin{lemma}\label{Lem.submerged}
\begin{equation}\label{eq:estimate1epsilon}
    I_{1,\varepsilon} \le \frac{p}{2}
  \left(\int_{\R^2}\rho_\varepsilon^p\right)^{(p-2)/p}
  \left(\int_{\R^2}
  |\nabla_Au|^p
  \right)^{2/p}+o(1)
  \qquad\mbox{and}\qquad
  I_{2,\varepsilon} = o(1).
\end{equation}
\end{lemma}
\begin{proof}

\ \medskip \\
\noindent 
\textbf{Estimate for $I_{1,\varepsilon}$:}
Integrating by parts and using the fact that $r_\varepsilon\le \rho_\varepsilon$ and that their gradients coincide, we get (as in the smooth case)
\begin{align*}
  \int_{\R^2}r_\varepsilon^p\curl V_\varepsilon
  =\int_{\R^2}r_\varepsilon^p
    (\partial_xV_\varepsilon^2-\partial_yV_\varepsilon^1) 
  =p\int_{\R^2}r_\varepsilon^{p-1}
    \nabla^\perp r_\varepsilon\cdot V_\varepsilon, 
\end{align*}
from which
\begin{equation*}
    I_{1,\varepsilon}\leq p\int_{\R^2}\rho_\varepsilon^{p-1}
    |\nabla\rho_\varepsilon|\,|V_\varepsilon|.
\end{equation*}
Applying Young's and H\"older's inequality 
as in the 
proof of Proposition~\ref{prop:low_bound_plane_regular}, 
we get
\begin{equation*}
  p\int_{\R^2}\rho_\varepsilon^{p-1}
  |\nabla\rho_\varepsilon|\,|V_\varepsilon|
  \leq \frac{p}{2}
  \left(\int_{\R^2}\rho_\varepsilon^p\right)^{(p-2)/p}
  \left(\int_{\R^2}
  \bigl(|\nabla\rho_\varepsilon|^2+
\rho_\varepsilon^2|V_\varepsilon|^2\bigr)^{p/2}
  \right)^{2/p}.
\end{equation*}
Furthermore,
\begin{align*}
|\nabla\rho_\varepsilon|^2+
\rho_\varepsilon^2|V_\varepsilon|^2
  = |\xi_\varepsilon|^2
  \bigl(|\nabla\rho_\varepsilon|^2+
\rho_\varepsilon^2|V_\varepsilon|^2\bigr)+(1-|\xi_\varepsilon|^2)\bigl(|\nabla\rho_\varepsilon|^2+
\rho_\varepsilon^2|V_\varepsilon|^2\bigr).
\end{align*}
Since 
\begin{equation*}
    1-|\xi_\varepsilon|^2= \frac{\varepsilon^2}{|u|^2+\varepsilon^2}= \frac{\varepsilon^2}{\rho^2_\varepsilon},
\end{equation*}
we have
\begin{equation*}
    |\nabla\rho_\varepsilon|^2+
\rho_\varepsilon^2|V_\varepsilon|^2
  = |\xi_\varepsilon|^2
  \bigl(|\nabla\rho_\varepsilon|^2+
\rho_\varepsilon^2|V_\varepsilon|^2\bigr)+ \varepsilon^2\bigg(\frac
{|\nabla \rho_\varepsilon|^2}{\rho_\varepsilon^2}+ |V_\varepsilon|^2\bigg).
\end{equation*}
Therefore, considering its $L^{p/2}$-norm, with $p>2$, we get by Minkowski inequality
\begin{equation*}
\begin{aligned}
  \bigg(\int_{\R^2}
  \bigl(|\nabla\rho_\varepsilon|^2+&
  \rho_\varepsilon^2|V_\varepsilon|^2\bigr)^{p/2}
  \bigg)^{2/p} \\
  & \leq
  \left(\int_{\R^2}|\xi_\varepsilon|^p
  \bigl(|\nabla\rho_\varepsilon|^2+
  \rho_\varepsilon^2|V_\varepsilon|^2\bigr)^{p/2}
  \right)^{2/p}
  +\varepsilon^2\left(\int_{\R^2}
  \left(\frac{|\nabla\rho_\varepsilon|^2}{\rho_\varepsilon^2}
  +|V_\varepsilon|^2\right)^{p/2}
  \right)^{2/p}.
\end{aligned}
\end{equation*}
We want to prove that the last integral tends to zero. 
Using
\begin{equation*}
  \nabla\rho_\varepsilon
  =\frac{\operatorname{Re}(\overline u\nabla u)}{\sqrt{|u|^2+\varepsilon^2}},
  \qquad
  |\nabla\rho_\varepsilon|\leq
  \frac{|u|\,|\nabla u|}{\sqrt{|u|^2+\varepsilon^2}},
\end{equation*}
we get
\begin{equation*}
    \frac{|\nabla\rho_\varepsilon|^2}{\rho_\varepsilon^2}\le \frac{|u|^2|\nabla u|^2}{(|u|^2+\varepsilon^2)^2}.
\end{equation*}
Furthermore
\begin{equation*}
  |V_\varepsilon|^2= \frac{|\operatorname{Im}(\overline u\nabla u)+A|u|^2|^2}{(|u|^2+\varepsilon^2)^2}
  \leq 2\left(
  \frac{|\operatorname{Im}(\overline u\nabla u)|^2+|A|^2|u|^4}{(|u|^2+\varepsilon^2)^2}
  \right)\le 2\left(
  \frac{|u|^2(|\nabla u|^2+|A|^2|u|^2)}{(|u|^2+\varepsilon^2)^2}\right).
\end{equation*}
Hence
\begin{equation*}
    \varepsilon^2\bigg(\frac{|\nabla\rho_\varepsilon|^2}{\rho_\varepsilon^2}
  +|V_\varepsilon|^2\bigg)\le \frac{\varepsilon^2|u|^2}{(|u|^2+\varepsilon^2)^2}(3|\nabla u|^2+|A|^2|u|^2).
\end{equation*}
Moreover, calling $t = (|u|/\varepsilon)^2$, it is easy to check that
\begin{equation*}
    \frac{\varepsilon^2|u|^2}{(|u|^2+\varepsilon^2)^2}= \frac{|u|^2}{\varepsilon^2(1+\frac{|u|^2}{\varepsilon^2})^2}= \frac{t}{(1+t)^2}\le \frac{1}{4}.
\end{equation*}
Eventually the family
\begin{equation*}
    g_\varepsilon = \bigg[\varepsilon^2\bigg(\frac{|\nabla\rho_\varepsilon|^2}{\rho_\varepsilon^2}
  +|V_\varepsilon|^2\bigg)\bigg]^\frac{p}{2}
\end{equation*}
is such that it pointwise converges to $0$. Then by the smoothness assumption and compact support of $u$ and the local boundedness of $A$, gives
\begin{equation*}
    g_\varepsilon\le \bigg[\frac{3|\nabla u|^2+|A|^2|u|^2}{4}\bigg]^\frac{p}{2}\in L^1(\mathbb R^2).
\end{equation*}
Therefore, by the dominated convergence theorem we have that 
\begin{equation*}
    \lim_{\varepsilon\to 0}\varepsilon^2\left(\int_{\R^2}
  \left(\frac{|\nabla\rho_\varepsilon|^2}{\rho_\varepsilon^2}
  +|V_\varepsilon|^2\right)^{p/2}
  \right)^{2/p}=0.
\end{equation*}
In this way, as $\varepsilon\to 0$, we have proved that 
\begin{equation*}
    \bigg(\int_{\R^2}
  \bigl(|\nabla\rho_\varepsilon|^2+
  \rho_\varepsilon^2|V_\varepsilon|^2\bigr)^{p/2}
  \bigg)^{2/p} \le \left(\int_{\R^2}|\xi_\varepsilon|^p
  \bigl(|\nabla\rho_\varepsilon|^2+
  \rho_\varepsilon^2|V_\varepsilon|^2\bigr)^{p/2}
  \right)^{2/p} + o(1).
\end{equation*}
From this inequality and recalling \eqref{eq:importantestimate}, 
we arrive to the first estimate of~\eqref{eq:estimate1epsilon}.

This estimate is crucial since it is independent on $V_\varepsilon$, which is the problematic term since for $\varepsilon\to 0$ we have
that
\begin{equation*}
    V_\varepsilon\to V_0 = A+ \frac{\operatorname{Im}(\overline u\nabla u)}{|u|},
\end{equation*}
which is not well defined on the set $\{u=0\}$.

\medskip
\noindent 
\textbf{Estimate for  $I_{2,\varepsilon}:$}  
We compute directly $\curl W_\varepsilon$. By the linearity of the $\curl$ operator, we have that
\begin{equation*}
    \curl W_\varepsilon= \curl\bigg(\frac{\operatorname{Im}(\overline{u}\nabla u)}{\rho_\varepsilon^2}\bigg)- \varepsilon^2\curl\bigg(\frac{A}{\rho_\varepsilon^2}\bigg).
\end{equation*}
Therefore we get
\begin{equation}\label{eq:curl1}
    \curl\bigg(\frac{\operatorname{Im}(\overline{u}\nabla u)}{\rho_\varepsilon^2}\bigg)= \frac{\curl(\operatorname{Im}(\overline{u}\nabla u))}{\rho_\varepsilon^2}- \nabla^\perp \bigg(\frac{1}{\rho_\varepsilon^2}\bigg)\cdot\operatorname{Im}(\overline{u}\nabla u), 
\end{equation}
and
\begin{equation}\label{eq:curl2}
    \curl\bigg(\frac{A}{\rho_\varepsilon^2}\bigg) = \frac{B}{\rho_\varepsilon^2}- \nabla^\perp \bigg(\frac{1}{\rho_\varepsilon^2}\bigg)\cdot A.
\end{equation}
Let us start by computing \eqref{eq:curl2}.
Applying \eqref{eq:A2.1}, we get
\begin{equation*}
    \nabla^\perp\bigg(\frac{1}{\rho_\varepsilon^2}\bigg) =  -\frac{\nabla^\perp(|u|^2)}{\rho_\varepsilon^4}=-  2\frac{\operatorname{Re}(\overline{u}\nabla^\perp u) }{\rho_\varepsilon^4},
\end{equation*}
and consequently
\begin{equation}
    \curl\bigg(\frac{A}{\rho_\varepsilon^2}\bigg) = \frac{B}{\rho_\varepsilon^2}+2\frac{\operatorname{Re}(\overline{u}\nabla^\perp u)\cdot A }{\rho_\varepsilon^4}.
\end{equation}
Let us compute now \eqref{eq:curl1}. By \eqref{eq:A2.3}, we have that
\begin{equation*}
    \frac{\curl(\operatorname{Im}(\overline{u}\nabla u))}{\rho_\varepsilon^2}= \frac{2 \det (\nabla\operatorname{Re}(u),\nabla\operatorname{Im}(u))}{\rho_\varepsilon^2}.
\end{equation*}
Again, using \eqref{eq:A2.1} and applying \eqref{eq:A2.2}, we get 
\begin{equation*}
    \nabla^\perp \bigg(\frac{1}{\rho_\varepsilon^2}\bigg)\cdot\operatorname{Im}(\overline{u}\nabla u) = -2\frac{\operatorname{Re}(\overline{u}\nabla^\perp u)\cdot \operatorname{Im}(\overline{u}\nabla u)}{\rho_\varepsilon^4}= \frac{|u|^2\det (\nabla\operatorname{Re}(u),\nabla\operatorname{Im}(u))}{\rho_\varepsilon^4}.
\end{equation*}
This gives us
\begin{equation*}
    \curl\bigg(\frac{\operatorname{Im}(\overline{u}\nabla u)}{\rho_\varepsilon^2}\bigg) = \bigg(1-\frac{|u|^2}{\rho^2_\varepsilon}\bigg)\frac{2}{\rho_\varepsilon^2}\det (\nabla\operatorname{Re}(u),\nabla\operatorname{Im}(u))= \frac{2\varepsilon^2}{\rho_\varepsilon^4}\det (\nabla\operatorname{Re}(u),\nabla\operatorname{Im}(u)).
\end{equation*}
Eventually
\begin{equation*}
\begin{split}
I_{2,\varepsilon}&=\bigg|\int_{\mathbb R^2}r_\varepsilon^p \curl W_\varepsilon\bigg| \le \int_{\mathbb{R}^2}r_\varepsilon^p\bigg|\curl\bigg(\frac{\operatorname{Im}(\overline{u}\nabla u)}{\rho_\varepsilon^2}\bigg)\bigg|+ \varepsilon^2\int_{\mathbb{R}^2}r_\varepsilon^p\bigg|\curl\bigg(\frac{A}{\rho_\varepsilon^2}\bigg)\bigg|\\
&= 2\int_{\mathbb{R}^2}\frac{\varepsilon^2r_\varepsilon^{p}}{\rho_\varepsilon^4}| \det (\nabla\operatorname{Re}(u),\nabla\operatorname{Im}(u))|+\int_{\mathbb{R}^2}\frac{\varepsilon^2r_\varepsilon^{p}}{\rho_\varepsilon^4}|B\rho_\varepsilon^2+2\operatorname{Re}(\overline{u}\nabla^\perp u)\cdot A |\\
&= \int_{\mathbb R^2}\frac{\varepsilon^2r_\varepsilon^{p}}{\rho_\varepsilon^4}\big(2| \det (\nabla\operatorname{Re}(u),\nabla\operatorname{Im}(u))|+|B\rho_\varepsilon^2+2\operatorname{Re}(\overline{u}\nabla^\perp u)\cdot A |\big).
\end{split}
\end{equation*}
The integrand converges pointwise to $0$. Moreover by the smoothness of $u$ and the assumptions on $A$, the term in the round bracket is bounded in the support of $u$. We only need to show that the term 
\begin{equation*}
  \frac{\varepsilon^2r_\varepsilon^{p}}{\rho_\varepsilon^4}
\end{equation*}
is bounded. Recalling~\eqref{regularisation}
and naming $r_\varepsilon= \varepsilon t$, we have
\begin{equation*}
  \frac{\varepsilon^2r_\varepsilon^{p}}{\rho_\varepsilon^4}= \frac{\varepsilon^2r_\varepsilon^{p}}{(r_\varepsilon+\varepsilon)^4}= \frac{\varepsilon^{p+2}t^p}{\varepsilon^4(1+t)}= \varepsilon^{p-2}\frac{t^p}{(1+t)^4}= \varepsilon^{p-2}\bigg(\frac{t}{(1+t)^2}\bigg)^2 t^{p-2}.
\end{equation*}
Now, as we did before we know that $t/(1+t)^2\le 1/4$, so that
\begin{equation*}
    \frac{\varepsilon^2r_\varepsilon^{p}}{\rho_\varepsilon^4}\le \frac{1}{16}(\varepsilon t)^{p-2}=\frac{r_\varepsilon^{p-2}}{16} \le \frac{|u|^{p-2}}{16}.
\end{equation*}
Hence this term is bounded whenever $p\ge 2$ (since $u$ is with compact support). By the dominated convergence theorem we have that $|I_{2,\varepsilon}|\to 0$ as $\varepsilon \to 0^+$.

\medskip
This concludes the proof of the submerged Lemma~\ref{Lem.submerged}.
\end{proof}
Coming back to~\eqref{submerged} and using Lemma~\ref{Lem.submerged},
we arrive to 
\begin{equation}\label{eq:estimatefinal}
    |B|\int_{\mathbb R^2}|u|^p \le I_{1,\varepsilon}+I_{2,\varepsilon}+ o(1)\le  \frac{p}{2}
  \left(\int_{\R^2}\rho_\varepsilon^p\right)^{(p-2)/p}
  \left(\int_{\R^2}
  |\nabla_Au|^p
  \right)^{2/p}+ o(1).
\end{equation}
Therefore, recalling that $\rho_\varepsilon\to |u|$ pointwise, again by the dominated convergence theorem we have that
\begin{equation*}
    \int_{\mathbb R^2}\rho_\varepsilon^p \to \int_{\mathbb R^2}|u|^p.
\end{equation*}
Hence, passing to the limit as $\varepsilon\to 0$ in \eqref{eq:estimatefinal}, we have
\begin{equation*}
    |B|\int_{\mathbb R^2}|u|^p \le  \frac{p}{2}
  \left(\int_{\R^2}|u|^p\right)^{(p-2)/p}
  \left(\int_{\R^2}
  |\nabla_Au|^p
  \right)^{2/p},
\end{equation*}
that in turn implies
\begin{equation*}
    \frac{\displaystyle\int_{\R^2}
  |\nabla_Au|^p}{\displaystyle\int_{\mathbb R^2}|u|^p}\ge \bigg(\frac{2}{p}\bigg)^{\frac{p}{2}}|B|^{\frac{p}{2}}.
\end{equation*}
This inequality has been proved for every $u\in C^\infty_c(\mathbb R^2,\mathbb C)$. 
By density, the result extends to $u \in W_{0,A}^{1,p}(\R^2)$.
The proof of Theorem~\ref{thm:lower-bound-plane} is thus concluded.

\subsection{The upper bound, proof of Proposition~\ref{cor:two-sided-plane}}
Now we turn to the upper bound for $\lambda_p^B(\mathbb R^2)$
contained in Proposition~\ref{cor:two-sided-plane}.

Without loss of generality,
we choose the symmetric gauge~\eqref{symmetric}.
    Let us use in the variational characterisation~\eqref{Rayleigh} 
    of $\lambda_p^B(\mathbb R^2)$ the first Landau eigenstate
    \begin{equation*}
        \psi(x) = e^{-\frac{|B|}{4}|x|^2}.
    \end{equation*}
 Since $\nabla \psi = -\frac{|B|}{2}xe^{-\frac{|B|}{4}|x|^2}$, then
    \begin{equation*}
        -i\nabla \psi + A\psi = \bigg(i\frac{|B|}{2}x+A\bigg)\psi.
    \end{equation*}
   Since $A$ is a real potential and $|A|^2=(|B||x|^2)/4$, then
    \begin{equation*}
        |\nabla_A \psi|^2 = \frac{|B||x|^2}{2}|\psi|^2.
    \end{equation*}
    This gives
    \begin{equation*}
        \int_{\mathbb R^2} |\nabla_A\psi|^p = \frac{|B|^p}{2^{\frac{p}{2}}}\int_{\mathbb R^2}|x|^p e^{-\frac{p|B|}{4}|x|^2}\,dx.
    \end{equation*}
    Using the coarea formula, we get
    \begin{equation*}
      \int_{\mathbb R^2}|x|^p e^{-\frac{p|B|}{4}|x|^2}\,dx= 2\pi \int_0^{+\infty}t^{p+1} e^{-\frac{p|B|}{4}t^2}\,dt.
    \end{equation*}
    Naming $r=\frac{p}{4}|B|t^2$, then
    \begin{equation*}
    \int_{\mathbb R^2}|x|^p e^{-\frac{p|B|}{4}|x|^2}\,dx=\bigg(\frac{4}{p|B|}\bigg)^\frac{p+2}{2}\pi \int_0^{+\infty}r^{\frac{p}{2}}e^{-r}\,dr = \bigg(\frac{4}{p|B|}\bigg)^\frac{p+2}{2}\pi \Gamma\bigg(\frac{p+2}{2}\bigg),
\end{equation*}
which gives
\begin{equation*}
  \int_{\mathbb R^2} |\nabla_A\psi|^p= \frac{2^{\frac{p+4}{2}}|B|^{\frac{p-2}{2}}}{p^\frac{p+2}{2}}\pi  \Gamma\bigg(\frac{p+2}{2}\bigg).
\end{equation*}
Now, applying again the coarea formula, we get
\begin{equation*}
    \int_{\mathbb R^2}|\psi|^p = 2\pi \int_0^{+\infty}te^{-\frac{p|B|}{4}t^2}\,dt= -\frac{4\pi}{p|B|}\int_0^{+\infty}\frac{d}{dt}\bigg(e^{-\frac{p|B|}{4}t^2}\bigg) \,dt = \frac{4\pi}{p|B|}.
\end{equation*}
Therefore
\begin{equation*}
    \lambda_p^B(\mathbb R^2) \le \Gamma\bigg(\frac{p+2}{2}\bigg)\bigg(\frac{2}{p}\bigg)^{\frac{p}{2}}|B|^{\frac{p}{2}}.
\end{equation*}
This concludes the proof of Proposition~\ref{cor:two-sided-plane}.

\subsection{The strong magnetic field, proof of Theorem~\ref{thm:strongregime}}
Finally, we prove Theorem~\ref{thm:strongregime} concerned with
the convergence of $\lambda_p^B(\Omega)$, for any $\Omega \subset \mathbb R^2$ open and simply connected set, in the strong magnetic field regime.

In view of the symmetry~\eqref{eq:B-B}, 
it is enough to consider the case
\(B\to+\infty\). Moreover, by the translation invariance \eqref{eq:gaugetranslations}, we may assume,
without loss of generality, that $0\in\Omega$. By the scaling property \eqref{eq:scalingproperty}, we get
\begin{equation*}
\frac{\lambda_p^B(\Omega)}
{\lambda_p^B(\mathbb R^2)}
=
\frac{\lambda_p^1(\sqrt{B}\,\Omega)}
{\lambda_p^1(\mathbb R^2)}.
\end{equation*}
Hence it is enough to prove that
\begin{equation*}
\lim_{B\to+\infty}
\lambda_p^1(\sqrt{B}\,\Omega)
=
\lambda_p^1(\mathbb R^2).
\end{equation*}
As we already pointed out, by the monotonicity with respect to the set inclusion, since $\sqrt{B}\Omega \subset \mathbb R^2$
\begin{equation*}
    \lambda_p^1(\sqrt{B}\Omega) \ge \lambda_p^1(\mathbb R^2).
\end{equation*}
Let us prove the inverse inequality. Let $\varepsilon>0$. By the definition of
$\lambda_p^1(\mathbb R^2)$, there exists
$u_\varepsilon\in C_c^\infty(\mathbb R^2,\mathbb C)$ such that $\|u_\varepsilon\|_{L^p(\mathbb R^2)}=1
$
and
\begin{equation*}
\int_{\mathbb R^2}
|\nabla_Au_\varepsilon|^p
\le
\lambda_p^1(\mathbb R^2)+\varepsilon.
\end{equation*}
Since \(u_\varepsilon\) has compact support, there exists a magnetic field $B(\varepsilon)$, such that for every $B\ge B(\varepsilon)$, we get
\begin{equation*}
\operatorname{supp}(u_\varepsilon)
\subset
\sqrt{B}\,\Omega.
\end{equation*} 
Consequently  $u_\varepsilon
\in
C_c^\infty(\sqrt{B}\,\Omega)$ and it can be used as a trial function in $\lambda_p^1(\sqrt{B}\Omega)$, giving
\begin{equation*}
\lambda_p^1(\sqrt{B}\,\Omega)
\le
\displaystyle
\int_{\mathbb R^2}
|\nabla_Au_\varepsilon|^p
\le
\lambda_p^1(\mathbb R^2)+\varepsilon.
\end{equation*}
The conclusion of Theorem~\ref{thm:strongregime}
follows by the arbitrariness of $\varepsilon$.

\section{Sharp geometric bounds}\label{sec:4}
In this section we give the proofs of the upper and lower bounds 
of Theorem~\ref{thm:polya-bounds} and Proposition~\ref{prop:lowboundpolya}, 
respectively. 
They represent an extension of the P\'olya-type bounds 
to the magnetic and nonlinear cases.

For the upper bound we will consider the affine family of magnetic gauges
\begin{equation}\label{affine}
A_\theta(x)
=
\big(-\theta  x_2,\,(1-\theta) x_1\big)B,
\end{equation}
with $\theta\in\mathbb R$, which satisfies
$
\curl A_\theta = B.
$
Since the proof of Theorem \ref{thm:polya-bounds} is quite long, we split it into two parts: in the first one we prove the inequalities~\eqref{eq:lambdapB} for $1<p\le 2$ and $p>2$, and in the other we prove their sharpness.


\subsection{Proof of Theorem \ref{thm:polya-bounds}:
Inequality \texorpdfstring{\eqref{eq:lambdapB}}{eq:lambdapB}}

In dimension $2$, any open and bounded convex set $\Omega$ can always be boxed in a rectangle of sides $w_\Omega$ and $P(\Omega)/2$, where $P(\Omega)$ and $w_\Omega$ denote the perimeter and the minimal width of the set $\Omega$ respectively. Therefore, we can always translate and rotate $\Omega$ in such a way
\begin{equation}\label{eq:rotationtranslation}
    |x_1|\le \frac{P(\Omega)}{4}, \qquad |x_2|\le \frac{w_\Omega}{2}, \qquad\forall x=(x_1,x_2)\in\Omega.
\end{equation}
Let $u$ be any real-valued trial function for $\lambda_p^B(\Omega)$, then since also $A_\theta$ is real, we get
\begin{equation}\label{eq:testfunct}
    \lambda_p^B(\Omega)\le\frac{\displaystyle\int_\Omega |(-i\nabla+A_\theta)u|^p\,dx}{\displaystyle\int_\Omega |u|^p}=\frac{\displaystyle\int_\Omega (|\nabla u|^2+|A_\theta|^2 |u|^2)^{\frac{p}{2}}\,dx}{\displaystyle\int_\Omega |u|^p}.
\end{equation}
Here we need to distinguish two cases. For the rest of the proof and in both cases, we will choose $u$ as a function depending only on the distance function from the boundary, i.e.
\[
u(x)=f(d(x)),\qquad f(0)=0,
\]
with $f$ real-valued.

\subsubsection*{Case $1< p\le 2$.} 

In this case we have that $\frac{1}{2}< \frac{p}{2}\le 1$, therefore by the concavity of the map $t\to t^\frac{p}{2}$, we have that $(x+y)^\frac{p}{2}\le x^\frac{p}{2}+y^\frac{p}{2}$, for $x,y \ge 0$, implying from \eqref{eq:testfunct} that
\begin{equation*}
    \lambda_p^B (\Omega) \le \frac{\displaystyle\int_\Omega|\nabla u|^p}{\displaystyle\int_\Omega |u|^p\,dx}+\frac{\displaystyle\int_\Omega |A_\theta|^p |u|^p\,dx}{\displaystyle\int_\Omega |u|^p\,dx}= I_1(u)+I_2(\theta,u). 
\end{equation*}

Let us start estimating the non-magnetic part $I_1$. By the coarea formula and the fact that $|\nabla d|=1$ a.e., we have
\[
\int_\Omega |\nabla u|^p
=
\int_0^{R_\Omega} |f'(t)|^p P(\Omega_t)\,dt,
\qquad
\int_\Omega |u|^p
=
\int_0^{R_\Omega} |f(t)|^p P(\Omega_t)\,dt.
\]

We now perform Pólya's change of variables
\[
s(t)=\frac{\pi_p}{2}\frac{\mu(t)}{|\Omega|},
\qquad s\in(0,\tfrac{\pi}{2}),
\]
so that
\[
\frac{ds}{dt}
=
-\frac{\pi_p}{2|\Omega|}P(\Omega_t).
\]
Setting $g(s)=f(t(s))$, we obtain
\begin{equation*}
    \frac{\displaystyle\int_\Omega |\nabla u|^p}{\displaystyle\int_\Omega |u|^p}
\le
\Big(\frac{\pi_p}{2}\Big)^p
\frac{1}{|\Omega|^p}
\frac{\displaystyle\int_0^{\pi_p/2}|g'(s)|^pP(t(s))^p\,ds}
{\displaystyle\int_0^{\pi_p/2}|g(s)|^2\,ds}.
\end{equation*}
Since $\Omega$ is convex, $P(\Omega_t)\le P(\Omega)$ for all $t$, hence
\[
\frac{\displaystyle\int_\Omega |\nabla u|^p}{\displaystyle\int_\Omega |u|^p}
\le
\Big(\frac{\pi_p}{2}\Big)^p
\frac{P(\Omega)^p}{|\Omega|^p}
\frac{\displaystyle\int_0^{\pi_p/2}|g'(s)|^p\,ds}
{\displaystyle\int_0^{\pi_p/2}|g(s)|^p\,ds}.
\]
Choosing the $p$-sine function $g(s)=\sin_p s$ (see \cite[Sec.~2.1]{DPP2019}), we obtain
\begin{equation}\label{eq:I1ple2}
I_1(u)=\frac{\displaystyle\int_\Omega |\nabla u|^p}{\displaystyle\int_\Omega |u|^p}
\le
\Big(\frac{\pi_p}{2}\Big)^p
\frac{P(\Omega)^p}{|\Omega|^p}.
\end{equation}
Let us now estimate the magnetic part $I_2$. 
Applying again coarea formula and the concavity inequality, we get
\begin{equation*}
\begin{split}
    \int_\Omega |A_\theta|^p |u|^p\,dx &=B^p \int_0^{R_\Omega}|f(t)|^p\int_{\partial \Omega_t}(\theta^2|x_2|^2+(1-\theta)^2|x_1|^2)^\frac{p}{2}\,d\mathcal{H}^1\,dt\\
    & \le B^p\bigg(\theta^p\int_0^{R_\Omega}|f(t)|^p\int_{\partial \Omega_t}|x_2|^p\,d\mathcal{H}^1+(1-\theta)^p\int_0^{R_\Omega}|f(t)|^p\int_{\partial \Omega_t}|x_1|^p\,d\mathcal{H}^1\bigg)\\
    &= B^p (\theta^p M_{2,p}+(1-\theta)^p M_{1,p}),
    \end{split}
\end{equation*}
where we have denoted by
\begin{equation}\label{eq:weightedpmoments}
    M_{i,p}= \int_0^{R_\Omega}|f(t)|^p\int_{\partial \Omega_t}|x_i|^p\,d\mathcal{H}^1, \qquad i=1,2,
\end{equation}
the weighted $p$-moments. Moreover, seeing it as a function of $\theta$, i.e., setting
\begin{equation*}
    F_p(\theta) =\theta^p M_{2,p}+(1-\theta)^p M_{1,p},
\end{equation*}
we have
\begin{equation*}
    F'_p(\theta) = p(\theta^{p-1} M_{2,p}-(1-\theta)^{p-1}M_{1,p}).
\end{equation*} 
Consequently,
$F_p'(\theta)=0$ if and only if
\begin{equation*}
    \bigg(\frac{\theta}{1-\theta}\bigg)^{p-1}= \frac{M_{1,p}}{M_{2,p}}\qquad \implies \qquad \theta_0 = \frac{M_{1,p}^{\frac{1}{p-1}}}{M_{1,p}^{\frac{1}{p-1}}+M_{2,p}^{\frac{1}{p-1}}}.
\end{equation*}
It is easy to check that $F''_p(\theta_0)>0$, then $\theta_0$ is the unique minimum point. Some straightforward computations gives
\begin{equation*}
    F_p(\theta_0 ) = \frac{M_{1,p}M_{2,p}}{\big(M_{1,p}^\frac{1}{p-1}+M_{2,p}^{\frac{1}{p-1}}\big)^{p-1}}.
\end{equation*}
Therefore, choosing $\theta=\theta_0$, we obtain
\begin{equation}\label{eq:AleBF}
    \int_\Omega |A_\theta|^p |u|^p\,dx\le B^pF_p(\theta_0).
\end{equation}
Firstly, let us stress that by \eqref{eq:rotationtranslation}, we have
\begin{equation}\label{eq:Mi}
M_{i,p}
\le
\max_{\Omega}|x_i|^p
\int_0^{R_\Omega} |f(t)|^p P(\Omega_t)\,dt\le\begin{cases}
\displaystyle\Big(\frac{P(\Omega)}{4}\Big)^p\int_\Omega |u|^p\,dx 
& \mbox{if} \quad i=1, \\
 \displaystyle   \Big(\frac{w_\Omega}{2}\Big)^p\int_\Omega |u|^p\,dx 
 & \mbox{if} \quad i=2.
\end{cases}
\end{equation}
Moreover it is clear that $F_p(\theta_0)\le M_{i,p}$ for $i=1,2$. This fact, together with \eqref{eq:AleBF} and \eqref{eq:Mi}, gives
\begin{equation*}
    I_2(\theta_0,u) = \frac{\displaystyle\int_\Omega |A_\theta|^p |u|^p\,dx}{\displaystyle\int_\Omega |u|^p\,dx}\le B^p\min\bigg\{\bigg(\frac{P(\Omega)}{4}\bigg)^p, \bigg(\frac{w_\Omega}{2}\bigg)^p\bigg\}.
\end{equation*}
Nevertheless, in the class of planar convex sets one has
$
    P(\Omega) \ge \pi w_\Omega,
$    
where the equality case is achieved on bodies with constant width. This implies that
\begin{equation*}
    \frac{P(\Omega)}{4}\ge \frac{\pi}{4}w_\Omega> \frac{w_\Omega}{2},
\end{equation*}
yielding
\begin{equation}\label{eq:I2ple2}
    I_2(\theta_0,u)\le \bigg(\frac{Bw_\Omega}{2}\bigg)^p.
\end{equation}
Inequalities \eqref{eq:I1ple2} and \eqref{eq:I2ple2} give a first estimate, that is
\begin{equation}\label{eq:firstestimatelambdaple2}
    \lambda_p^B(\Omega) \le \Big(\frac{\pi_p}{2}\Big)^p
\frac{P(\Omega)^p}{|\Omega|^p} +\bigg(\frac{Bw_\Omega}{2}\bigg)^p.
\end{equation}
Moreover, we can estimate $F(\theta_0)$ differently. 
Applying the arithmetic--geometric inequality   
\begin{equation*}
    \sqrt{xy}\le \bigg(\frac{x^r+y^r}{2}\bigg)^\frac{1}{r}, 
    \qquad r\ge 1, \quad x,y \ge 0, 
\end{equation*}
to the $p$-means with
$r = (p-1)^{-1}>1$, we get
\begin{equation*}
 M_{1,p}M_{2,p} \le  \frac{\big(M_{1,p}^{\frac{1}{p-1}}+M_{2,p}^{\frac{1}{p-1}}\big)^{2(p-1)}}{2^{2(p-1)}}.
\end{equation*}
In particular 
\begin{equation*}
    F_p(\theta_0) \le \frac{\big(M_{1,p}^{\frac{1}{p-1}}+M_{2,p}^{\frac{1}{p-1}}\big)^{p-1}}{2^{2(p-1)}}.
\end{equation*}
Therefore choosing $\theta = \theta_0$, we get

\begin{equation*}
    \int_\Omega |A_\theta|^p u^p\le B^p \frac{\big(M_{1,p}^{\frac{1}{p-1}}+M_{2,p}^{\frac{1}{p-1}}\big)^{p-1}}{2^{2(p-1)}}.
\end{equation*}
Now, since $1<p\le 2$, by the concavity of $t\to t^{p-1}$, we get
\begin{equation*}
    \big(M_{1,p}^{\frac{1}{p-1}}+M_{2,p}^{\frac{1}{p-1}}\big)^{p-1}\le M_{1,p}+M_{2,p}.
\end{equation*}
Using again \eqref{eq:Mi}, we arrive to
\begin{equation*}
    \int_\Omega |A_\theta|^p u^p\le \frac{B^p}{2^{2(p-1)}}\bigg[\bigg(\frac{P(\Omega)}{4}\bigg)^p+\bigg(\frac{w_\Omega}{2}\bigg)^p\bigg] \int_\Omega |u|^p\,dx,
\end{equation*}
implying 
\begin{equation}\label{eq:secondestimatelambdaple2}
    \lambda_1^B(\Omega) \le I_1(u)+I_2(u,\theta_0)\le \Big(\frac{\pi_p}{2}\Big)^p
\frac{P(\Omega)^p}{|\Omega|^p}+
\frac{B^p}{2^{2(p-1)}}\bigg[\bigg(\frac{P(\Omega)}{4}\bigg)^p+\bigg(\frac{w}{2}\bigg)^p\bigg]. 
\end{equation}
By \eqref{eq:firstestimatelambdaple2} and \eqref{eq:secondestimatelambdaple2} we have the thesis for the first case.

\subsubsection*{Case $p>2$.} 

In this case we elevate $\lambda_1^B(\Omega)$ to the power $\frac{2}{p}$ and use the Minkowski inequality, having
\begin{equation*}
    \begin{split}\lambda_1^B(\Omega)^{\frac{2}{p}} &\le \frac{\displaystyle\bigg(\int_\Omega (|\nabla u|^2+|A_\theta|^2 |u|^2)^{\frac{p}{2}}\,dx\bigg)^{\frac{2}{p}}}{\displaystyle\bigg(\int_\Omega |u|^p\,dx \bigg)^\frac{2}{p}}\\
    &\le \left(\frac{\displaystyle\int_\Omega |\nabla u|^p\,dx}{\displaystyle\int_\Omega |u|^p\,dx }\right)^\frac{2}{p}+\left(\frac{\displaystyle\int_\Omega |A_\theta|^p |u|^p\,dx}{\displaystyle\int_\Omega |u|^p\,dx }\right)^\frac{2}{p}= I_1(u)^\frac{2}{p}+I_2(\theta,u)^\frac{2}{p} .
    \end{split}
\end{equation*}
As before
\begin{equation}\label{eq:I1pge2}
    I_1(u)^\frac{2}{p} \le \left(\Big(\frac{\pi_p}{2}\Big)^p
\frac{P(\Omega)^p}{|\Omega|^p}\right)^\frac{2}{p}= \Big(\frac{\pi_p}{2}\Big)^2
\frac{P(\Omega)^2}{|\Omega|^2}.
\end{equation}
We only need to estimate $I_2(\theta,u).$ In this case we can apply another time Minkowski inequality to obtain
\begin{equation*}
    \begin{split}
\left(\int_\Omega |A_\theta|^p |u|^p\,dx\right)^\frac{2}{p}&= B^2\left(\int_\Omega (\theta^2|x_2|^2|u|^2+(1-\theta)^2|x_1|^2|u|^2)^\frac{p}{2}\right)^\frac{2}{p}\\
&\le B^2 \left(\int_\Omega (\theta^p|x_2|^p|u|^p\,dx \right)^\frac{2}{p}+     \left(\int_\Omega ((1-\theta)^p|x_2|^p|u|^p\,dx \right)^\frac{2}{p}\\
&= B^2(\theta^2 M_{2,p}^\frac{2}{p}+ (1-\theta)^2M_{1,p}^\frac{2}{p}),
    \end{split}
\end{equation*}
where $M_{i,p}$ has been defined in \eqref{eq:weightedpmoments}. We define
\begin{equation*}
    G_p(\theta)=\theta^2 M_{2,p}^\frac{2}{p}+ (1-\theta)^2M_{1,p}^\frac{2}{p},
\end{equation*}
and, as in the previous case, it is possible to prove that its unique minimum of $G_p(\theta)$ is given by
\begin{equation*}
    G_p(\theta_0) = \frac{M_{1,p}^\frac{2}{p}M_{2,p}^\frac{2}{p}}{M_{1,p}^\frac{2}{p}+M_{2,p}^\frac{2}{p}}. 
\end{equation*}
Since $G_p(\theta_0)\le M_{i,p}^{\frac{2}{p}}$, following the same argument as before, we have that
\begin{equation}\label{eq:I2pge2}
I_2(\theta_0,u)^\frac{2}{p}\le \bigg(\frac{Bw_\Omega}{2}\bigg)^2.
\end{equation}
Therefore using \eqref{eq:I1pge2}--\eqref{eq:I2pge2} and elevating to the power $p/2$, we get 
\begin{equation}\label{eq:firstestimatelambdapge2}
    \lambda_p^B(\Omega) \le \bigg[\Big(\frac{\pi_p}{2}\Big)^2
\frac{P(\Omega)^2}{|\Omega|^2}+\bigg(\frac{Bw_\Omega}{2}\bigg)^2\bigg]^\frac{p}{2}.
\end{equation}
Let us give another estimate, now. Using again the classical arithmetic--geometric mean inequality, we get
\begin{equation*}
    G_p(\theta_0) \le \frac{M_{1,p}^\frac{2}{p}+M_{2,p}^\frac{2}{p}}{4}.
\end{equation*}
Therefore, applying again \eqref{eq:Mi}, we have that
\begin{equation}\label{eq:I2pge22}
    I_2(\theta,u)^\frac{2}{p} \le \frac{B^2}{4} \bigg[\bigg(\frac{P(\Omega)}{4}\bigg)^2+ \bigg(\frac{w_\Omega}{2}\bigg)^2\bigg].
\end{equation}
This gives 
\begin{equation*}
    \lambda_1^B(\Omega) \le \bigg\{\Big(\frac{\pi_p}{2}\Big)^2
\frac{P(\Omega)^2}{|\Omega|^2}+\frac{B^2}{4} \bigg[\bigg(\frac{P(\Omega)}{4}\bigg)^2+ \bigg(\frac{w_\Omega}{2}\bigg)^2\bigg]\bigg\}^\frac{p}{2},
\end{equation*}
and therefore the thesis in the case $p>2$.

\subsection{Proof of Theorem \ref{thm:polya-bounds}: 
Sharpness of inequality \texorpdfstring{\eqref{eq:lambdapB}}{eq:lambdapB}}
The sharpness of the inequality is understood in the isoperimetric sense: Although equality is never attained, the constant $(\pi_p/{2})^p$ is optimal, as shown by a family of thinning domains.
Let $0<a<1$ be a real number and let us consider the rectangle $\Omega_a = \big(0,1\big)\times \big(0,a\big)$. Let us stress that if $a$ is sufficiently small, then
\begin{equation*}
    \min\bigg\{w_{\Omega_a}^p,2^{2-p}\bigg[\bigg(\frac{P(\Omega_a)}{4}\bigg)^p+\bigg(\frac{w_{\Omega_a}}{2}\bigg)^p\bigg]\bigg\}= w_{\Omega_a}^p, \qquad  \min\bigg\{w_{\Omega_a}^2, \bigg(\frac{P(\Omega_a)}{4}\bigg)^2+ \bigg(\frac{w_{\Omega_a}}{2}\bigg)^2\bigg\}= w_{\Omega_a}^2.
\end{equation*}
Using the diamagnetic inequality~\eqref{diamagnetic} 
and \cite[Lem.~3.1]{MR4393129}, we have that
\begin{equation*}
    \lambda_p^B(\Omega_a)\ge \lambda_p^0(\Omega_a)\ge \bigg(\frac{a}{2}\bigg)^{-p}\bigg(\frac{\pi_p}{2}\bigg)^p.
\end{equation*}
Therefore, with the new upper bound we know that
\begin{equation*}
    \bigg(\frac{a}{2}\bigg)^{-p}\bigg(\frac{\pi_p}{2}\bigg)^p\le \lambda_p^B(\Omega_a)\le\begin{cases}
      \displaystyle\bigg(\frac{\pi_p}{2}\bigg)^p\frac{P(\Omega_a)^p}{|\Omega_a|^p}
+
\bigg(\frac{B w_{\Omega_a}}{2}\bigg)^p 
& \mbox{if} \quad 1<p\le 2, \\ \\
\displaystyle\bigg[\bigg(\frac{\pi_p}{2}\bigg)^2\frac{P(\Omega_a)^2}{|\Omega_a|^2}
+
\bigg(\frac{B w_{\Omega_a}}{2}\bigg)^2\bigg]^\frac{p}{2} 
& \mbox{if} \quad p>2.
    \end{cases} 
\end{equation*}
Multiplying by $|\Omega_a|^p/P(\Omega_a)^p$, we get
\begin{equation*}
    \frac{|\Omega_a|^p}{P(\Omega_a)^p}\bigg(\frac{a}{2}\bigg)^{-p}\bigg(\frac{\pi_p}{2}\bigg)^p\le \frac{\lambda_p^B(\Omega_a)|\Omega_a|^p}{P(\Omega_a)^p}\le
    \begin{cases}
        \displaystyle\bigg(\frac{\pi_p}{2}\bigg)^p
+
\frac{|\Omega_a|^p}{P(\Omega_a)^p}\bigg(\frac{B w_{\Omega_a}}{2}\bigg)^p & \mbox{if} \quad 1<p\le 2, \\ \\
\displaystyle
\bigg[\bigg(\frac{\pi_p}{2}\bigg)^2
+
\frac{|\Omega_a|^2}{P(\Omega_a)^2}\bigg(\frac{B w_{\Omega_a}}{2}\bigg)^2\bigg]^\frac{p}{2} & \mbox{if} \quad p>2.
    \end{cases}
\end{equation*}
It is clear that
\begin{equation*}
    P(\Omega_a) = 2(1+a),\qquad |\Omega_a|=a,\qquad w_{\Omega_a}= a,
\end{equation*}
therefore
\begin{equation*}
    \frac{|\Omega_a|^p}{P(\Omega_a)^p}\bigg(\frac{a}{2}\bigg)^{-p}\bigg(\frac{\pi_p}{2}\bigg)^p=\bigg(\frac{2a}{2(1+a)}\bigg)^p\bigg(\frac{\pi_p}{2}\bigg)^p\to\bigg(\frac{\pi_p}{2}\bigg)^p\qquad \text{as}\;\;a\to 0^+.
\end{equation*}
On the other side, when $1<p\le 2$:
\begin{equation*}
    \bigg(\frac{\pi_p}{2}\bigg)^p+\frac{|\Omega_a|^p}{P(\Omega_a)^p}\bigg(\frac{B w_{\Omega_a}}{2}\bigg)^p= \bigg(\frac{\pi_p}{2}\bigg)^p+B^p \bigg(\frac{a^2}{4(1+a)}\bigg)^p\to \bigg(\frac{\pi_p}{2}\bigg)^p \qquad \text{as}\;\;a\to 0^+,
\end{equation*}
while for $p>2$
\begin{equation*}
    \bigg[\bigg(\frac{\pi_p}{2}\bigg)^2+\frac{|\Omega_a|^2}{P(\Omega_a)^2}\bigg(\frac{B w_{\Omega_a}}{2}\bigg)^2\bigg]^\frac{p}{2}= \bigg[\bigg(\frac{\pi_p}{2}\bigg)^2+B^2 \bigg(\frac{a^2}{4(1+a)}\bigg)^2\bigg]^\frac{p}{2}\to \bigg(\frac{\pi_p}{2}\bigg)^p \qquad \text{as}\;\;a\to 0^+.
\end{equation*}
Therefore for every $p>1$, have
\begin{equation*}
    \lim_{a\to 0^+}\frac{\lambda_p^B(\Omega_a)|\Omega_a|^2}{P(\Omega_a)^2}= \bigg(\frac{\pi_p}{2}\bigg)^p,
\end{equation*}
proving the sharpness of the inequality for families of thinning rectangles $\Omega_a$.

\begin{oss}
   The presence of the minimum in the inequality \eqref{eq:lambdapB} takes into count somehow the ``thickness'' of the set. Indeed, let us consider the linear case $p=2$. We have that
   \begin{equation*}
     \lambda_2^B(\Omega)
\le
\frac{\pi^2}{4}\frac{P(\Omega)^2}{|\Omega|^2}
+
\frac{B^2}{4}\min \bigg\{w^2_\Omega, \frac{P(\Omega)^2}{16}+\frac{w^2_\Omega}{4}\bigg\}. 
   \end{equation*}
   We note that 
   \begin{equation*}
       \frac{P(\Omega)^2}{16}+\frac{w^2_\Omega}{4}= w^2_\Omega\quad\iff\quad  \frac{P(\Omega)}{w_\Omega}= 2\sqrt{3},
   \end{equation*}
and
   \begin{equation*}
       \min \bigg\{w_\Omega^2, \frac{P(\Omega)^2}{16}+\frac{w_\Omega^2}{4}\bigg\} =\begin{cases}
       \displaystyle\frac{P(\Omega)^2}{16}+\frac{w_\Omega^2}{4} 
       & \displaystyle
       \text{when}\;\; \frac{P(\Omega)}{w_\Omega}\le 2\sqrt{3},\\
       \\
       w_\Omega^2 
       & \displaystyle 
       \text{when}\;\; \frac{P(\Omega)}{w_\Omega}\ge 2\sqrt{3}.
       \end{cases}
   \end{equation*}
  We stress that $2\sqrt{3}$ is the exact ratio between the perimeter and the minimal width of a equilateral triangle.
   From this observation we can say that if a planar convex set $\Omega$ is such that $P(\Omega)/w_\Omega< 2\sqrt{3}$, that it is thicker than the equilateral triangle in the sense of the perimeter-to-width ratio.
   Indeed if, for instance $\Omega = B_1$ is the open unit ball, then 
   \begin{equation*}
         \frac{P(\Omega)^2}{16}+\frac{w_\Omega^2}{4} =\frac{4\pi^2}{16}+1< 4 = w_\Omega^2.
   \end{equation*}
\end{oss}

\subsection{Proof of Proposition \ref{prop:lowboundpolya}}
Finally, we give a proof of 
the lower bound of Proposition \ref{prop:lowboundpolya}. 
This time we do not split the proof of the inequality and the proof of the sharpness, since the former follows immediately by the diamagnetic inequality \eqref{diamagnetic} and Hersch-Protter inequality.

  To prove the inequality, it suffices to recall that 
  the diamagnetic inequality~\eqref{diamagnetic} gives
  $
      \lambda_1^B(\Omega) \ge \lambda_p^0(\Omega).
  $    
 Now, recalling the Hersch--Protter inequality (see \cite{hersch1960frequence,protter1981lower} for $p=2$ and \cite{brasco_inradius} for $p\neq 2$), which states that
  \begin{equation}\label{eq:HerschProtter}
      \lambda_p(\Omega)R_\Omega^p \ge \bigg(\frac{\pi_p}{2}\bigg)^p,
  \end{equation}
  and using the following inequality proved in \cite{fenchel_bonnesen}
  \begin{equation}\label{eq:bonnesen2d}
      1 \le \frac{P(\Omega)R_\Omega}{\abs{\Omega}}\le 2,
      \end{equation}
      then 
      \begin{equation*}
          \frac{\lambda_p^B(\Omega)\abs{\Omega}^p}{P(\Omega)^p}\ge \frac{\lambda_p^0(\Omega)\abs{\Omega}^p}{P(\Omega)^p}= \lambda_p^0(\Omega)R_\Omega^p \frac{\abs{\Omega}^p}{P(\Omega)^pR_\Omega^p} \ge \bigg(\frac{\pi_p}{4}\bigg)^p.
      \end{equation*}
      
  To prove the sharpness, we follow \cite[Thm.~4.2]{brasco_inradius}, with the necessaries differences due to the presence of the magnetic field. Let $0<\alpha<1$ be a positive parameter and let us consider the point $V_\alpha= (0,\alpha)\in \mathbb R^2$. Let us denote by $T_\alpha$ the isosceles triangle given by
  $
      T_\alpha = conv\{(-1,1), V_\alpha\}.
  $    
  First, let us stress that 
  $
      T_\alpha \subset \mathbb R^2\times (0,\alpha), 
  $    
  So that by the monotonicity of $\lambda^B_p(\cdot)$ with respect to the set inclusion,
  \begin{equation}\label{eq:lowboundstripe}
      \lambda_p^B(T_\alpha) \ge \lambda_p^B(\mathbb R\times (0,\alpha) ) \ge \lambda_p^0(\mathbb R\times (0,\alpha) ) = \bigg(\frac{\pi_p}{\alpha}\bigg)^p.
  \end{equation}
  Let us now prove a reverse-type inequality. 
  If we consider the rectangle   $Q_\alpha\subset T_\alpha$ given by
  \begin{equation*}
  Q_\alpha = (-(1-\sqrt{\alpha}),1-\sqrt{\alpha})\times (0,\alpha(1-\sqrt{\alpha})) = \alpha(1-\sqrt{\alpha}) \;[ (-\alpha^{-1},\alpha^{-1})\times(0,1)],
  \end{equation*}
  then by the monotonicity property with respect to the set inclusion and the scaling property~\eqref{eq:scalingproperty} we get
  \begin{equation}\label{eq:uppboundstripe}
      \lambda_p^B(T_\alpha)\le \lambda_p^B(Q_\alpha) = [\alpha(1-\sqrt{\alpha})]^{-p}\lambda_p^{B_\alpha}\bigg((-\alpha^{-1},\alpha^{-1})\times(0,1)\bigg),
  \end{equation}
  where $B_\alpha = \alpha^2(1-\sqrt{\alpha})^2B$. We know that 
  \begin{equation*}
      \lim_{B\to 0^+}\lambda_p^B(\Omega)= \lambda_p^0(\Omega).
  \end{equation*}
  Therefore, if we let $\alpha \to 0^+$, then $B_\alpha \to 0^+$, so that 
  \begin{equation}\label{eq:limiteigstripe}
      \lim_{\alpha \to 0^+}\lambda_p^{B_\alpha}\bigg((-\alpha^{-1},\alpha^{-1})\times(0,1)\bigg) = \lambda_p^0\bigg(\mathbb R\times(0,1)\bigg)= \pi_p^p.
  \end{equation}
  Therefore \eqref{eq:lowboundstripe}, \eqref{eq:uppboundstripe} and \eqref{eq:limiteigstripe} imply
  \begin{equation}\label{eq:lambdapBisosceles}
      \lim_{\alpha \to 0^+}\lambda_p^B(T_\alpha) \bigg(\frac{\pi_p}{\alpha}\bigg)^{-p}=1.   \end{equation}
      Let us recall that
      \begin{equation}\label{eq:inradiusisosceles}
          R_{T_\alpha} = \frac{\alpha}{1+\sqrt{1+\alpha^2}}.
      \end{equation}
      Equations \eqref{eq:lambdapBisosceles} and \eqref{eq:inradiusisosceles} imply
      \begin{equation*}
          \lim_{\alpha \to 0^+} \lambda_p^B(T_\alpha)R_{T_\alpha}^p =\bigg( \frac{\pi_p}{2}\bigg)^p. 
      \end{equation*}
      Moreover it is easy to verify that
      \begin{equation*}
          \frac{P(T_\alpha)R_{T_\alpha}}{\abs{T_\alpha}} = 2,
      \end{equation*}
      so that
      \begin{equation*}
          \lim_{\alpha \to 0^+} \frac{\lambda_p^B(T_\alpha)\abs{T_\alpha}^p}{P(T_\alpha)^p} = \lim_{\alpha \to 0^+}\lambda_p^B(T_\alpha)R_{T_\alpha}^p \cdot\frac{\abs{T_\alpha}^p}{P(T_\alpha)^pR_{T_\alpha}^p} = \bigg(\frac{\pi_p}{4}\bigg)^p,
      \end{equation*}
      proving the sharpness of the inequality.

\section{Open problems}\label{sec:5}

The results obtained in this paper naturally suggest several directions for further investigation.

\begin{enumerate}

\item 
In Theorem~\ref{thm:lower-bound-plane},
we proved for $p\ge 2$ the lower bound
\begin{equation*}
\lambda_p^B(\mathbb{R}^2)\ge \bigg(\frac{2}{p}\bigg)^\frac{p}{2}|B|^{\frac{p}{2}},
\end{equation*}
and we know that for $p=2$ the equality holds. 
We leave as an open problem whether 
there is actually an equality even if $p>2$. Furthermore, does there exist a minimiser in the variational
characterisation~\eqref{Rayleigh} of $\lambda_p^B(\mathbb{R}^2)$?
Recall that it is the case in the linear case $p=2$,
where $\lambda_2^B(\mathbb{R}^2)$ is actually an eigenvalue
of infinite multiplicity.

\item The technique used to prove the lower bound 
of Theorem~\ref{thm:lower-bound-plane}
cannot be applied to the case $p < 2$. This fact is visible in Proposition \ref{prop:low_bound_plane_regular}, where  H\"older inequality fails to apply. Moreover the hypothesis $p\ge 2$ is necessary when applying the dominated convergence theorem in the proof of Theorem~\ref{thm:lower-bound-plane}, 
when we want the integral $I_{2,\varepsilon}$ to go to zero. Therefore an open problem is to find a constant $C(p)>0$ such that for every $1\le p <2$, we get
\begin{equation*}
    \lambda_p^B(\mathbb R^2)\ge C(p)|B|^\frac{p}{2}.
\end{equation*}
Since the upper bound in \ref{cor:two-sided-plane} holds for any $p\ge 1$, the only information we know is that $C(p)\le \Gamma\big(\frac{p+2}{2}\big)$ and that $C(p)$ cannot be $(2/p)^\frac{p}{2}$, since for $p\in [1,2)$ we have that $\Gamma\big(\frac{p+2}{2}\big)< 1$. 

\item Does Theorem~\ref{thm:lower-bound-plane} hold 
in the case of a non-constant magnetic field?
Recall that it is the case in the linear case $p=2$

\item
In this paper, we typically assume $1<p<\infty$,
as usual for the $p$-Laplacian.
In Proposition~\ref{cor:two-sided-plane}, however,
we managed to include the limiting case $p=1$, too.
Inspired by the Cheeger problem,
it would be interesting to state a geometric version 
of the perimeter-over-area minimisation in the magnetic case, too. 
Note that the necessary theory of magnetic BV-functions 
was developed in~\cite{PSV2019}. 

\item
Interesting spectral-geometric question 
particularly arise in unbounded tubes, 
recently extended to the case of the $p$-Laplacian in~\cite{BK5}.
The diamagnetic effect of the magnetic field 
is quantified by Hardy-type inequalities in the linear case $p=2$,
see~\cite{KR}. The nonlinear case remains open.

\end{enumerate}

\saveformat
\appendix

\section*{Acknowledgements}
The authors were supported  by the grant no. 26-21940S
of the Czech Science Foundation. R.S. has been partially supported by GNAMPA group of INdAM.

\section*{Conflicts of interest and data availability statement}
The authors declare that there is no conflict of interest. Data sharing not applicable to this article as no datasets were generated or analysed during the current study.

\bibliographystyle{plain}
\bibliography{references02}
\Addresses
\end{document}